\documentclass[11pt]{article}
\usepackage{amssymb,amsmath,amsthm,amsfonts,mathrsfs}
\usepackage{subfigure,graphicx}
\usepackage{ulem}
\graphicspath{{pictureSimulation/}}
\usepackage{color}
\usepackage{verbatim}
\usepackage{float, tikz, array, caption}
\usetikzlibrary{shapes,arrows}

\def\disp{\displaystyle}

\usepackage[CJKbookmarks, colorlinks, bookmarksnumbered=true,pdfstartview=FitH,linkcolor=black]{hyperref}
\theoremstyle{plain}
\newtheorem{theorem}{Theorem}[section]
\newtheorem{lemma}{Lemma}[section]

\newtheorem{corollary}{Corollary}[section]

\numberwithin{equation}{section}
\theoremstyle{definition}
\newtheorem{remark}{Remark}[section]

\def\crr{\cr\noalign{\vskip2mm}}

\def\dref#1{(\ref{#1})}
\newcommand{\R}{{\mathbb R}}

\def\H{{\cal H}}

\begin{document}
\title{{\bf
%Output Feedback Stabilization   of Unstable  Heat Equation with   Unknown   Input Gain
  A  New Adaptive    Control Scheme    for  Unstable  Heat Equation with   Unknown   Control  Coefficient\footnote{\small This work is supported by the National Natural Science Foundation of China (No. 62273217, 12131008, 62373231) and the Fundamental Research Program of Shanxi Province (202203021223002).} }}
% A New Nonlinear Dynamic Feedback Approach to Disturbance Treatment for Infinite-Dimensional Systems.
\author{ Hongyinping Feng\footnote{\small  Corresponding author.
Email: fhyp@sxu.edu.cn.}  \ and\  \  Hai-Li Du
\\
{\it \small School of Mathematical Sciences}
{\it \small Shanxi University,  Taiyuan, Shanxi, 030006, P.R.
China}\\
 }
%\date{}

\maketitle
\begin{abstract}% Abstract of not more than 200 words.
In this paper, we develop a novel  and simple
 adaptive control scheme
  for  a one-dimensional   unstable  heat equation with unknown control coefficient.
 A   new state observer  is designed to estimate the system state, while    a new    update law  is devised to   estimate the
   reciprocal of the
 control coefficient.
In contrast with the conventional state observer   which is usually available for  all   controllers, the newly designed state observer
  depends on a special controller  factorization.   Very importantly,
  both  the unknown control coefficient and its corresponding  estimate do not appear in the state observer anymore.
Consequently, the   stabilization of
  the control plant   comes down to  the stabilization of
  the state observer.
 In this way, the obstacles caused  by the unknown control coefficient can be overcome thoroughly.
As an application,
the performance  output
  tracking
 % for the unstable heat equation with unknown control coefficient
  is also considered by the newly developed   approach. When the reference signal is persistently exciting,
  the reciprocal of the control
coefficient can be estimated effectively by the  designed       update law.
   All the aforementioned  results are  proved mathematically, although the resulting closed-loop system may be nonlinear. Some of them are  validated visually by numerical simulations.

\end{abstract}
\vspace{0.3cm}
\vspace{0.3cm}

 \section{Introduction}

Systems    with unknown control coefficient
are   ubiquitous in many fields of engineering,
 even for the infinite-dimensional system.
However,  compared with the   works  on the finite-dimensional system   in literature,  the problem caused by unknown  control coefficient   is rarely considered
for the   infinite-dimensional system.
% although it is very important and common in
%engineering applications.
To the best of our knowledge, the   estimation of unknown control coefficient  in unstable infinite-dimensional systems is still  a challenging problem until now.  It is quite difficult  to
 achieve the desired system performance when the
  control coefficient is unknown.
 In this  paper, we    present an adaptive approach to address  the  problem caused by unknown control coefficient
through the  following  benchmark unstable heat system:
 \begin{equation}\label{20237122035}
\left\{\begin{array}{l}
w_{t}(x,t)=  w_{xx}(x,t)  , \; x\in (0,1),
\; t> 0,\crr
 w_{x} (0,t)=-qw(0,t) ,\ \
w_{x}(1,t)=bu(t) ,  \   \  t\ge 0, \ \ q>0,\crr
 y (t)=  (w(0,t),w(1,t))   ,\ \ t\ge 0,
 \end{array}\right.
\end{equation}
where  $w $ is the system state,
%$c>0$ is represented the heat conductivity,
 $u$ is the  control input,
   $y $ is the measurement  output and   $b\neq0$ is the unknown control coefficient.
 The model \dref{20237122035} is a general one-dimensional  heat equation with boundary
convection. It depicts the flow of heat in a rod that is insulated
everywhere except the two ends.
   The convection at the left end is  proportional   to the temperature $w(0,t)$. It may destroy  the system stability.
   The   actuator with  control coefficient $b$ is installed on the right end $x=1$.
 The control  coefficient $b$      depends  usually  on
many factors such as   the specific heat capacity,   the area of cross-section,  the density as well as the thermal conductivity of   the rod.
   It may be unknown  in the practical applications.
  The   more details of physical modeling of heat equation can be found in \cite{Hahn2012}.

System \dref{20237122035}  is unstable provided $q>1$ \cite{Andrey4}.
%  has infinitely many poles  and one of the poles is unstable provided $q>1$.
 %  and is approximately equal to $q2$ as $q \to+\infty$.
The unstable  heat systems    have been extensively studied in recent years.
Some representative works can be found in \cite{Liuw}, \cite{backsteppingheat2004}, \cite{backsteppingheat2005},
\cite{Deutscher2},  \cite{Fengheat}  and the monograph \cite{Backsteppingbook}, where
 the  partial differential equation backstepping  method   has been used to cope with the
  unstable term.
    Although the method of backstepping is powerful,
  it relies   on the precise information
 on the control coefficient  which
may  be unknown in practical applications.
 Therefore, how to stabilize the infinite-dimensional systems,  including the unstable heat system \dref{20237122035},   with
  unknown control  coefficient  is still an unsolved  issue.

 In this paper, we will address the problems that is resulted in  the unknown control coefficient
 by the adaptive control approach.
 A new update law is  devised to estimate   the reciprocal of $b$
rather than   the control coefficient $b$  itself.
Meanwhile, a new  state observer,  that does not rely on the
  control coefficient and its estimation,      is designed to estimate the system state effectively.  We emphasize  that the
   state estimation    does not depend  on the convergence of
 control coefficient estimation.
%Thanks to our the state observer,
 As a consequence of this good characteristic of the state observer,
  %By designing an observer where the control coefficient dose not appear anymore, the
 a  new adaptive control scheme is  developed to stabilize the  unstable heat
equation \dref{20237122035} even if
  the  control coefficient $b$  is  unknown.
Additionally, if the controller meets the
so called persistent
excitation condition, the unknown control coefficient
can be estimated effectively   by
the newly designed  update law.
We   point out that the introduction of adaptive control approach  in the controller design for the  infinite-dimensional unstable system is   not our original creation.
In \cite{Krs2}, \cite{Andrey3}, \cite{Andrey4}     and the  references therein,
the adaptive control approach,
  together with the  backstepping technique, has been used successfully to
  deal with the heat equation with  unknown  unstable coefficient.

% One of the major obstacles in developing adaptive
%schemes for PDEs with unknown   control coefficient      is
%
%To overcome this obstacle,

%
%the absence of parametrized families of
%stabilizing controllers

%
%Note that the above works mainly relied on the precise information
%on the control coefficients

%  When the
% control coefficient is unknown in  a control system,
%the conventional approaches such as backstepping can not be used directly.
%
%1 the importance of input gain $b$
%
%2, Adaptive in PDE.

%Note that the above works mainly relied on the precise information
%on the control coefficients and the growth rate of nonlinearity
%during the control scheme. However, these information
%may always be unknown in practical applications, therefore how to
%cope with the global stabilization of nonlinear systems with these
%unknowns is a hot issue. To overcome this obstacle,

%We consider the system \dref{20237122035} in the state space $\H=L^2(0,1)$
%equipped with an inner product
%\begin{equation}\label{2023921954}
%\langle f,g\rangle_{\H}=\int_{0}^{1}f(x)g(x)dx,\ \ \forall\ f,g\in \H.
%\end{equation}

We proceed as follows. In Section  \ref{Obs}, we present the update law for the      reciprocal of $b$ and present the state observer for system \dref{20237122035}.
 An observer-based output  feedback stabilizer  is designed   in Section \ref{Obserstabilizer}.
As an application of the results obtained in  Sections \ref{Obs} and \ref{Obserstabilizer}, we consider the performance output tracking for system \dref{20237122035} in  Section \ref{Tracking}.
 The main results in
  Sections \ref{Obs}, \ref{Obserstabilizer} and \ref{Tracking} are
  proved strictly  in    Sections \ref{PfTh1},  \ref{PfTh2} and \ref{PfTh3}, respectively.
 The   mathematical proofs arranged in these sections
can be skipped if the reader is interested only in the design
procedure.
 Some numerical simulations are presented  in Section \ref{NumSim} to validate the theoretical results,
followed by some conclusions in Section \ref{Concluding}.
A  mathematical result  that
is less relevant to the feedback or observer design is  arranged
in Appendix.

For the sake of
simplicity, we drop the obvious temporal  and spatial domains in the
rest of the paper.
 Throughout the article,  $\H$  denotes  the Hilbert space
   $ L^2(0,1)$  which is equipped with an inner product
  \begin{equation}\label{2023962248}
\langle f,g\rangle_{\H}=\int_{0}^{1}f(x)g(x)dx,\ \ \forall\ f,g\in \H.
\end{equation}
  % For simplicity, we denote  $\|\cdot\|_{\infty}=\|\cdot\|_{L^{\infty}(0,\infty)}$.
We denote  the  set  of functions  $W$ as
  \begin{equation}\label{202473949}
W =\{f+g\ |\ f\in L^{\infty}(0,\infty), g\in L^{2}(0,\infty)\}.
  \end{equation}

\section{Observer  design}\label{Obs}
When $b$ is known,  the following controller can stabilize  system \dref{20237122035} exponentially
 by using   the method of  partial differential equation backstepping \cite{Backsteppingbook}:
 \begin{equation}\label{2023813835}
 \left.\begin{array}{l}
 \disp u(t)=-\frac{q+c_0}{b}\left[  w(1,t)+q \int_{0}^{1}e^{q(1-x)}w(x,t)dx\right],
 % \hat{w}_{t}(x,t)=\hat{w}_{xx}(x,t)   , \crr
% \hat{w}_{x} (0,t)=-q {w}(0,t),\  \
% \hat{w}_{x}(1,t)= u (t)+c_1 [w(1,t)-\hat{w}(1,t)] ,
 \end{array}
 \right.
 \end{equation}
  where $c_0$ is a positive  tuning parameter.
However, when $b$ is unknown, the situation becomes completely different.
  The unknown control coefficient  $b$ may  give  rise to  great difficulties  in
  %  both  the observer and
     the  controller   design.
For example, the conventional approach to
the controller design may be infeasible
 even if we have known the estimation
 \begin{equation}\label{2023719933}
 \hat{b}(t)\to b\ \ \mbox{as}\ \ t\to+\infty.
 \end{equation}
Actually,
 in terms of   \dref{2023719933} and \dref{2023813835},  the controller is designed naturally as
   \begin{equation}\label{2023719934}
 u(t)=-\frac{q+c_0}{\hat{b}(t)}\left[ w(1,t)+q \int_{0}^{1}e^{q(1-x)}w(x,t)dx\right].
 \end{equation}
  However,  the controller \dref{2023719934} is meaningless at time $t$  if  $\hat{b}(t)=0$ which may take place and is hard to avoid  in the convergence \dref{2023719933}.
  Moreover, the unknown  control coefficient  can also  affect the estimation/cancellation of the disturbance   in the well-known Active Disturbance Rejection Control \cite{FengAnnual} and the  solvability of
the regulator equations    in the Internal Model
Principle  \cite{Natarajan2016TAC}, \cite{PauLassi2016TAC} and \cite{PauLassi2017SIAM}.
In summary, we will encounter    great  difficulties
  in the presence of
    unknown control coefficient.

  In order to  overcome the obstacles  caused by  the  unknown control coefficient $b$, we  propose a novel strategy that   estimates     the reciprocal     of  $ b$ rather than $b$ itself.
  Additionally,
    we  factorize the controller  by
  \begin{equation}\label{2023719937}
 u(t)=\zeta(t)u_0(t),
 \end{equation}
where  $\zeta $ is used to compensate for the unknown $b$  and
$u_0 $ is a new controller.
 Both of them  will  be determined later.
  Thanks to the new controller type  \dref{2023719937}, the observer of system \dref{20237122035} can be designed   as
 \begin{equation}\label{20237122036}
\left\{\begin{array}{l}
% u(t)=\zeta(t)u_0(t),\crr
\hat{w}_{t}(x,t)=\hat{w}_{xx}(x,t)   ,
 \crr
 \hat{w}_{x} (0,t)=-q {w}(0,t),\  \
 %-c_0[w(0,t)-\hat{w}(0,t)] ,\ \ c_0>0,\crr
\hat{w}_{x}(1,t)= u_0(t)+c_1 [w(1,t)-\hat{w}(1,t)] ,   \crr
 \dot{\zeta}(t)= -\mbox{sgn}(b)[w(1,t)-\hat{w}(1,t)]u_0(t) ,
 \end{array}\right.
\end{equation}
 where $  c_1  $ is a positive      tuning gain. The observer \dref{20237122036} consists of two parts: the state observer, i.e., $\hat{w}$-subsystem,   and the update law for  $\zeta$.
  Please note that both  the unknown control coefficient $b$ and the newly designed  $\zeta$  do not   appear in the state observer   anymore.
 This  is
   precisely the subtlety of controller factorization  \dref{2023719937} and is
    very important for  the controller design.
   Let the observation error be
 \begin{equation}\label{20237122039frac1b}
 \tilde{w} (x,t)=w(x,t)-\hat{w}(x,t),\ \ \tilde{\zeta}(t)=\frac{1}{b}-  \zeta(t).
 \end{equation}
Then it is governed by
 \begin{equation}\label{20237122040frac1b}
\left\{\begin{array}{l}
\tilde{w}_{t}(x,t)=\tilde{w} _{xx}(x,t)   ,
 \crr
 \tilde{w} _{x} (0,t)= 0 ,\ \ \tilde{w} _{x}(1,t)=-b\tilde{\zeta}(t)u_0(t) -c_1\tilde{w} (1,t),    \crr
\dot{\tilde{\zeta}}(t)=  \mbox{sgn}(b) u_0(t)\tilde{w} (1,t).
 \end{array}\right.
\end{equation}
By   Lemmas \ref{Lm2023828938} and \ref{Lm2023811150} in Section \ref{PfTh1} below, we can  prove that the error system \dref{20237122040frac1b}
  is   asymptotically
   stable   for any $u_0\in  W$ and any initial state
  $(\tilde{w}(\cdot,0),\tilde{\zeta}(0))\in \H\times \R$.   Actually,
if we choose the  Lyapunov          function of error system \dref{20237122040frac1b}
as
\begin{equation}\label{202310142044}
 V(t)=\frac12\int_{0}^{1}\tilde{w}^2(x,t)dx+\frac{|b|}{2}\tilde{\zeta}^2(t),
\end{equation}
a       formal       computation will show  that $$\dot{V}(t)=-\int_{0}^{1}\tilde{w}^2_x(x,t)dx-c_1\tilde{w}^2(1,t)\leq0.$$
 The new design of  the update law   for  $\zeta$ in observer \dref{20237122036} is based on the  Lyapunov          function \dref{202310142044}.

  %We consider the observer \dref{20237122036} in the Hilbert space $\H=L^2(0,1)\times \R$.
  \begin{theorem}\label{Th20238281955}
Let $q>0$, $b\neq 0$,  $c_1>0$ and  $u_0\in W $. Suppose that the sign of $b$ is known. Then the observer \dref{20237122036} of system \dref{20237122035}
is well-posed: for any initial state $(w(\cdot,0),\hat{w}(\cdot,0),\zeta(0))\in \H^2\times \R$,
the observer  \dref{20237122036} admits a unique solution
$(\hat{w}(\cdot,t), {\zeta}(t))\in C(0,\infty;\H\times\R)$
such that
 \begin{equation}\label{20238282000}
 |{\zeta} (t) -\zeta_0|+\|w(\cdot,t)-\hat{w}(\cdot,t)\|_{\H} \to 0 \ \ \mbox{as}\ \ t\to+\infty,
 \end{equation}
%and
%\begin{equation}\label{20238282001}
% {\zeta} (t)  \to \zeta_0 \ \ \mbox{as}\ \ t\to+\infty ,
% \end{equation}
where $\zeta_0$ is a constant that may not be $\frac1b$.
If we suppose further  that the controller $u_0$ meets the following  persistent excitation condition:
 there exists a time $\tau>0$ such that
\begin{equation}\label{202381809*}
      \lim_{t\to+\infty}\int_{t}^{t+\tau}
       u_0 (s) ds \neq0,
\end{equation}
 then
\begin{equation}\label{20238282005}
   \zeta (t)  \to  \frac1b \ \ \mbox{as}\ \ t\to+\infty .
 \end{equation}

\end{theorem}

%\begin{remark}\label{Re202310141756}
%The design of   update law  for the    reciprocal of control coefficient $b$ in the observer \dref{20237122036} is inspired by the
%standard idea of adaptive control for  finite-dimensional system.
%\end{remark}

\begin{remark}\label{Re20238282005}
By \dref{20237122035} and \dref{2023719937},  we have $w_x(1,t)=bu(t)=b\zeta(t)u_0(t)$.  Although
the observer \dref{20237122036} is well-posed,
the new control coefficient $b\zeta(t)$ of $u_0(t)$ remains unknown until estimation \dref{20238282005}.
Consequently, the design of $u_0$ may still remain challenging.
 This implies that the separation principle based on the observer   \dref{20237122036} cannot be directly applied to the output feedback design of system \dref{20237122035}, even if full state estimation \dref{20238282000} is available.
On the other hand, according to Theorem \ref{Th20238281955}, the state estimation in \dref{20238282000} holds true as long as $u_0 \in W $.  Consequently, we can stabilize the control plant  \dref{20237122035}
via stabilizing the state observer \dref{20237122036}.
   Given that the coefficient of controller $u_0$ in the state observer in \dref{20237122036}  always equals 1, stabilizing the state observer becomes straightforward and feasible. This feature is particularly significant for observer-based controller designs when the control coefficient $b$ is unknown.

\end{remark}

\begin{remark}\label{Re2023914}
By the  conventional definition  \cite{PE1987}, the signal $u_0$ satisfies the  persistent
excitation condition    means that
there exist  constants  $\tau $, $t_0 $ and $c_0$ such that
   \begin{equation}\label{202381917}
 \int_{t }^{t+\tau}|u_0(s)|^2ds
 \geq  c_0>0,\ \ \forall\ t> t_0\geq 0,\ \ \tau>0.
\end{equation}
Clearly, the  function $u_0$ that meets    \dref{202381917}    always meets
the  persistent
excitation condition \dref{202381809*}.
% On the other hand, since $L^2(0,\infty) \subset\!\!\!\!\!\!\!/ \  L^1(0,\infty)$, we can choose $\gamma\in L^2(0,\infty)$ such that $\gamma\notin L^1(0,\infty)$.
%Therefore,
%for any $\tau>0$, we can conclude that
% \begin{equation}\label{2023914929}
%\lim_{t\to+\infty} \int_{t }^{t+\tau}|\gamma(s)|^2ds=0.
%\end{equation}
%%which implies that $\gamma$ is not persistent excitation in the sense \dref{202381917}.
% However,
%owing to $\gamma\notin L^1(0,\infty)$, it follows from the Cauchy convergence criteria  that  there exists a $\tau_0>0$ such that
% \begin{equation}\label{2023914937}
%\lim_{t\to+\infty} \int_{t }^{t+\tau_0}|\gamma(s)| ds\neq0.
%\end{equation}
%If we set $u_0(t)=|\gamma(t)|$, then the function $u_0$ meets the our assumption \dref{202381809*} but does not meet    \dref{202381917}.
Therefore, our  assumption \dref{202381809*} is not stronger      than the conventional persistent
excitation condition \dref{202381917}.
\end{remark}

\section{Observer-based stabilizer}\label{Obserstabilizer}

This section is devoted to the observer-based stabilizer design for the system \dref{20237122035}.
Taking  the observation error \dref{20237122039frac1b} into account, the state observer in \dref{20237122036} can be rewritten as
   \begin{equation}\label{20238312133}
\left\{\begin{array}{l}
% u(t)=\zeta(t)u_0(t),\crr
\hat{w}_{t}(x,t)=\hat{w}_{xx}(x,t)   ,
 \crr
 \hat{w}_{x} (0,t)=-q \hat{w}(0,t)-q \tilde{w}(0,t),\crr
\hat{w}_{x}(1,t)= u_0(t)+c_1  \tilde{w}(1,t)  .
% \dot{\zeta}(t)= -\mbox{sgn}(b) \tilde{w}(1,t) u_0(t) ,
 \end{array}\right.
\end{equation}
 %where $  c_1>0$ is  the    tuning gain.
 By Theorem \ref{Th20238281955}, $  \|w(\cdot,t)-\hat{w}(\cdot,t)\|_{\H}  \to 0$  as $t\to+\infty$ provided $u_0\in W$.  In this case,  the
 stabilization of  system \dref{20237122035}  can come down to  stabilization of  system
 \dref{20238312133}.  Since  $\| \tilde{w}(\cdot,t)\|_{\H}  \to 0$   and the unknown control coefficient $b$ does not appear in the observer \dref{20238312133},
we can stabilize the observer \dref{20238312133} easily
  by   ignoring the
 $\tilde{w}$-parts. Indeed, in view of the feedback \dref{2023813835}, the  controller $u_0$ can be designed as
 \begin{equation}\label{20237122039}
u_0(t)= -(q+c_0)  \left[  \hat{w}(1,t)+q \int_{0}^{1}e^{q(1-x)}\hat{w}(x,t)dx\right],
 \end{equation}
under
 which we get the  closed-loop system of system \dref{20237122035}
\begin{equation}\label{2023720848}
\left\{\begin{array}{l}
w_{t}(x,t)=w_{xx}(x,t)  ,\crr
 w_{x} (0,t)=-qw(0,t) ,\crr
\disp w_{x}(1,t)=-b(q+c_0)\zeta(t) \left[  \hat{w}(1,t)+q \int_{0}^{1}e^{q(1-x)}\hat{w}(x,t)dx\right], \crr
  \hat{w}_{t}(x,t)=\hat{w}_{xx}(x,t)   ,\crr
 \hat{w}_{x} (0,t)=-q {w}(0,t),\crr
\disp  \hat{w}_{x}(1,t)= -(q+c_0) \left[  \hat{w}(1,t)+q \int_{0}^{1}e^{q(1-x)}\hat{w}(x,t)dx\right]+c_1 [w(1,t)-\hat{w}(1,t)]  , \crr
 \disp\dot{\zeta}(t)=  \mbox{sgn}(b)(q+c_0) [w(1,t)-\hat{w}(1,t)] \left[  \hat{w}(1,t)+q \int_{0}^{1}e^{q(1-x)}\hat{w}(x,t)dx\right]  .
 \end{array}\right.
\end{equation}
 % where $c_0$ and $c_1$ are positive tuning parameters.
%We consider the closed-loop system  \dref{2023720848}  in the Hilbert space %$\mathcal{X}=L^2(0,1)\times L^2(0,1)\times \R$ with inner product

 \begin{theorem}\label{th20237141510}
Let $q>0$, $b\neq 0$,  $c_0>0$ and $c_1>0$. Suppose that the sign of $b$ is known. Then, for any initial state
$(w(\cdot,0),\hat{w}(\cdot,0),\zeta(0))\in \H^2\times \R$,
the closed-loop system  \dref{2023720848}
  admits a weak  solution
$(w(\cdot,t), \hat{w}(\cdot,t), {\zeta}(t))\in C(0,\infty;\H^2\times\R)$ such that
  \begin{equation}\label{2023720844w}
 |\zeta(t)- \zeta_0|+\|(w(\cdot,t) ,\hat{w}(\cdot,t))\|_{\H^2} \to 0 \ \ \mbox{as}\ \ t\to+\infty,
 \end{equation}
%and
%\begin{equation}\label{2023720844zeta}
%%\sup_{t\geq0}| \zeta(t) |<+\infty.
%   \zeta(t)\to \zeta_0\ \ \mbox{as}\ \ t\to+\infty,
% \end{equation}
where $\zeta_0$ is a  constant that may not be $\frac1b$.

 \end{theorem}

%Since  we still  do not known whether or not  the PE condition \dref{202381809*}   is satisfied before the controller design, the unknown input gain still hard to  be  compensated.
%This means that the  observer-based controller design is still not straightforward
%even if we have known the full state feedback and the estimation of the full state.

 \begin{remark}\label{Re2023831}
 In Theorem \ref{th20237141510}, we have not verified the persistent excitation condition \dref{202381809*} for the function $u_0$ in \dref{20237122039}. In fact, the stability of the closed-loop system \dref{2023720848} always implies that the function $u_0$ is not persistently exciting. This   will be proved in Lemma \ref{Lm2023951644} and demonstrated visually  through numerical simulations in Section \ref{NumSim}.
\end{remark}

%\begin{remark}\label{Re202392}
%Applied to wave
%\end{remark}

\section{Performance output tracking}\label{Tracking}
In this section, we apply the theoretical results in Sections \ref{Obs} and \ref{Obserstabilizer} to the performance output tracking for system \dref{20237122035}.
In this way, the corresponding controller can  be   persistently
excited   and hence the unknown control coefficient $b$  can be estimated effectively.
We aim at   designing
a   feedback control such  that the performance output $w(0,t) $ satisfies
\begin{equation}\label{2023421108}
 \int_0^{\infty}| w(0,t) - r(t)|^2dt<+\infty,
  \end{equation}
where $r\in C^{\infty}[0,+\infty)$  is a given reference.
Performance output tracking for infinite-dimensional system has been  studied  extensively in recent years.   However, to the  best of our knowledge,
the result about     the infinite-dimensional systems with unknown control coefficient  is
still  fairly scarce.
All the    approaches  that have been used in  \cite{Natarajan2016TAC}, \cite{PauLassi2016TAC}, \cite{PauLassi2017SIAM},
\cite{GuoWAdaptiveAuto},   \cite{GuoZhaoRX2022}  and \cite{Guomeng}
rely on the  precise information on the control coefficient.
 Moreover,   in contrast with the  aforementioned results where the reference is assumed to be generated by an exosystem,
   the considered reference $r$   is more general and is non-collocated to the controller. This
  also brings some  difficulties to the controller design.

To achieve  the desired output tracking, we need to
  construct the servo dynamics.
Inspired by \cite{105} and \cite[Chapter 12]{Backsteppingbook}, we let
 \begin{equation}\label{2023829858}
 v(x,t)=\sum_{j=0}^{\infty}r^{(j)}(t)\left[\frac{x^{2j}}{(2j)!}-\frac{ qx ^{2j+1}}{(2j+1)!}\right], \ \ x\in[0,1],\ \ t\geq0.
 \end{equation}
% where
% \begin{equation}\label{2016571824}
% \left.\begin{array}{l}
% \disp v_1(x,t)=\sum_{j=0}^{\infty}r^{(j)}(t)\frac{x^{2j}}{(2j)!}
% \end{array}\right.
%\end{equation}
%and
% \begin{equation}\label{20194251113}
% \disp v_2(x,t)=\sum_{j=0}^{\infty}r^{(j)}(t)\frac{ qx ^{2j+1}}{(2j+1)!}.
% \end{equation}
  By a simple computation,
    the function $v$ is governed by
 \begin{equation}\label{20221131932}
 \left\{\begin{array}{l}
 \disp v_t(x,t)=v_{xx}(x,t),\crr
\disp v_x(0,t)=-qr(t),\ \ v(0,t)=r(t)
   \end{array}\right.
\end{equation}
and furthermore,
  \begin{equation}\label{2023828927}
v_x(1,t) =  -qr(t)-\sum_{j=1}^{\infty}r^{(j)}(t)\left[\frac{q}{(2j)!} -\frac{1}{(2j-1)!}\right].
\end{equation}
 If we let
 \begin{equation}\label{2023829902}
  z(x,t)=w(x,t)-v(x,t), \ \ x\in[0,1],\ \ t\geq0,
\end{equation}
then
\begin{equation}\label{2023829903}
  z(0,t)=w(0,t)-v(0,t)=w(0,t)-r(t), \   \ t\geq0,
\end{equation}
and hence
 \begin{equation}\label{20221131937}
 \left\{\begin{array}{l}
  \disp  z_t(x,t)=z_{xx}(x,t),   \crr
  \disp  z_x(0,t)=
    -qz(0,t) , \ z_x(1,t)=bu(t)-v_x(1,t).
 \end{array}\right.
\end{equation}
It
follows from \dref{2023829902}  and  \dref{2023829858} that
 \begin{equation}\label{2023828z1}
z(1,t) = w(1,t)-v(1,t)=w(1,t)- \sum_{j=0}^{\infty}r^{(j)}(t)\left[\frac{1}{(2j)!} -\frac{q }{(2j+1)!}  \right].
\end{equation}
Owing to \dref{2023829903}, we are able to achieve the output tracking \dref{2023421108} as long as we can stabilize system \dref{20221131937}.
Inspired by the observer and stabilizer  designs in Sections  \ref{Obs} and \ref{Obserstabilizer}, the observer of system \dref{20221131937} can be designed as
\begin{equation}\label{20238291031}
  \left\{\begin{array}{l}
  u(t)=\zeta(t)u_0(t),\crr
 \hat{z}_{t}(x,t)=\hat{z}_{xx}(x,t)   , \; x\in (0,1),
 \crr
 \hat{z}_{x} (0,t)=-q {z}(0,t),\  \
\hat{z}_{x}(1,t)= u_0(t)-v_x(1,t)+c_1 [z(1,t)-\hat{z}(1,t)] ,    \crr
 \dot{\zeta}(t)= -\mbox{sgn}(b)[z(1,t)-\hat{z}(1,t)]u_0(t) ,
        \end{array}\right.
\end{equation}
%Owing to \dref{2023829903}, the output tracking comes down to stabilize system \dref{20221131937}.
%Furthermore, we just need to stabilize the observer \dref{20238291031}, as considered in the
% Section \ref{Obserstabilizer}.
%If we can prove $\|z(\cdot,t)-\hat{z}(\cdot,t)\|_{\H}^2\to 0$ as $t\to+\infty$,
%Inspired by the stabilizer design in Section \ref{Obserstabilizer},
and the observer-based  controller is designed   as
  \begin{equation}\label{20221111734}
  \begin{array}{ll}
   u_0(t)&\disp =   \disp  -(q+c_0)\left[  \hat{z}(1,t)+q \int_{0}^{1}e^{q(1-x)}\hat{z}(x,t)dx\right]+ v_x(1,t),
  % &\disp =   \disp  -\frac{q+c_0}{b}\left[  z(1,t)+q \int_{0}^{1}e^{q(1-x)}z(x,t)dx\right]-\frac1b   \sum_{j=0}^{\infty}r^{(j)}(t)\left[\frac{q^{2j}}{(2j)!} +\frac{1}{(2j-1)!}+qr(t) \right] ,
        \end{array}
\end{equation}
where $c_0  $  and $c_1$ are positive     tuning parameters.
%Inspired by the stabilizer design in Section \ref{Obserstabilizer}, the controller with unknown input gain can be designed as
% \begin{equation}\label{2023828934}
%  \left\{\begin{array}{l}
%  u(t) =\zeta(t)u_0(t),\crr
%   u_0(t) \disp =   \disp  -(q+c_0)\left[  \hat{z}(1,t)+q \int_{0}^{1}e^{q(1-x)}\hat{z}(x,t)dx\right]-\frac1b v_x(1,t),\crr
%\hat{z}_{t}(x,t)=\hat{z}_{xx}(x,t)   , \; x\in (0,1),
% \crr
% \hat{z}_{x} (0,t)=-q {z}(0,t),\  \
%\hat{z}_{x}(1,t)= u_0(t)+c_1 [z(1,t)-\hat{z}(1,t)] ,    \crr
% \dot{\zeta}(t)= \mbox{sgn}(b)[z(1,t)-\hat{z}(1,t)]u_0(t) ,
%        \end{array}\right.
%\end{equation}
%where $c_1>0$ is a  tuning parameter.
 Combining \dref{2023828z1}, \dref{20238291031}, \dref{20221111734} and  \dref{2023829903},  we obtain the closed-loop system of \dref{20237122035}:
 \begin{equation}\label{2023829945}
\left\{\begin{array}{l}
w_{t}(x,t)=w_{xx}(x,t)  ,\crr
 w_{x} (0,t)=-qw(0,t) ,\ \
\disp w_{x}(1,t)= b\zeta(t)u_0(t) , \crr
  u_0(t) \disp =   \disp  -(q+c_0)\left[  \hat{z}(1,t)+q \int_{0}^{1}e^{q(1-x)}\hat{z}(x,t)dx\right]+ v_x(1,t),\crr
\hat{z}_{t}(x,t)=\hat{z}_{xx}(x,t)   , \; x\in (0,1),
 \crr
 \hat{z}_{x} (0,t)=-q [w(0,t)-r(t)],\  \
\hat{z}_{x}(1,t)= u_0(t)+c_1 [w(1,t)-v(1,t)-\hat{z}(1,t)] -v_x(1,t),    \crr
 \dot{\zeta}(t)= -\mbox{sgn}(b)[w(1,t)-v(1,t)-\hat{z}(1,t)]u_0(t) ,
 \end{array}\right.
\end{equation}
  where
  % $c_0$ and $c_1$ are positive tuning parameters and
  $v(1,t)$ and $v_x(1,t)$ are given by
  \dref{2023829858} and \dref{2023828927}, respectively.

 \begin{theorem}\label{th2023829953}
Let $q>0$, $b\neq 0$,  $c_0>0$ and $c_1>0$. Suppose that the sign of $b$ is known and 
%there exists a positive constant $M_r$ such that
the reference $r\in C^{\infty}[0,\infty)$  satisfies
  \begin{equation}\label{2023829957}
\sup_{t\geq0, j=0,1,\cdots} | r^{(j)}(t)|<+\infty.
 \end{equation}
 Then for any initial state
$(w(\cdot,0),\hat{z}(\cdot,0),\zeta(0))\in  \H^2\times\R$,
the closed-loop system  \dref{2023829945}
  admits a  weak solution
$(w(\cdot,t), \hat{z}(\cdot,t), {\zeta}(t))\in C(0,\infty;\H^2\times\R)$ such that
  \begin{equation}\label{2023829954}
  \int_0^{\infty}|w(0,t)-r(t)|^2dt<+\infty
 \end{equation}
and
\begin{equation}\label{2023829956}
\sup_{t\geq0}\left[\|w(\cdot,t)\|_{\H}+\|\hat{z}(\cdot,t)\|_{\H}+| b\zeta(t)-1 |\right]<+\infty.
 \end{equation}
If we assume further that the reference $r$ satisfies the following    persistent excitation
condition:
\begin{equation}\label{202396903}
\lim_{t\to+\infty}\int_t^{t+\tau}v_x(1,s)ds\neq0,
 \end{equation}
%where  $v_x(1,t) $ is given by \dref{2023828927},
then
\begin{equation}\label{202396902}
  \zeta(t)-\frac1b \to 0 \ \ \mbox{as}\ \ t\to+\infty.
 \end{equation}

 \end{theorem}
%\begin{example}\label{Ex2023829956}

   When $r(t)\equiv r^*$ with $ r^* \neq0$.
   By \dref{2023828927}, $v_x(1,t) =  -qr(t)\equiv-qr^*$
and hence the  persistent excitation condition \dref{202396903} obviously  holds.
In this case, the
  closed-loop system \dref{2023829945} is reduced to
 \begin{equation}\label{2023829945constant}
\left\{\begin{array}{l}
w_{t}(x,t)=w_{xx}(x,t)  ,\crr
 w_{x} (0,t)=-qw(0,t) ,\ \
\disp w_{x}(1,t)= b\zeta(t)u_0(t) , \crr
  u_0(t) \disp =   \disp  -(q+c_0)\left[  \hat{z}(1,t)+q \int_{0}^{1}e^{q(1-x)}\hat{z}(x,t)dx\right]- qr^*,\crr
\hat{z}_{t}(x,t)=\hat{z}_{xx}(x,t)   , \; x\in (0,1),
 \crr
 \hat{z}_{x} (0,t)=-q [w(0,t)-r^*],\  \
\hat{z}_{x}(1,t)= u_0(t)+c_1 [w(1,t)- (1-q)r^*-\hat{z}(1,t)] +qr^*,    \crr
 \dot{\zeta}(t)= -\mbox{sgn}(b)[w(1,t)-(1-q)r^*-\hat{z}(1,t)]u_0(t) .
 \end{array}\right.
\end{equation}
 % where $c_0$ and $c_1$ are positive tuning parameters.
  Theorem \ref{th2023829953} leads immediately  to the following corollary:

\begin{corollary}\label{Co202396}
Suppose that  $q>0$, $b\neq 0$,  $c_0>0$,  $c_1>0$, $ r^*\neq0$ and the sign of $b$ is known.
 Then for any initial state
$(w(\cdot,0),\hat{z}(\cdot,0),\zeta(0))\in  \H^2\times\R$,
the closed-loop system  \dref{2023829945constant}
  admits a  weak solution
$(w(\cdot,t), \hat{z}(\cdot,t), {\zeta}(t))\in C(0,\infty;\H^2\times\R)$ such that
  \dref{2023829954}, \dref{2023829956} and \dref{202396902}  hold with $r(t)\equiv r^*$.
 \end{corollary}

%  \begin{proof}
%    Since    $v_x(1,t) =  -qr(t)\equiv r^*$, the     persistent
%excitation condition \dref{202396903} holds.
%  \end{proof}

%\begin{remark}\label{Re20231092013}
%  Owing to the space limitation, we have not prove the uniqueness of  solution  to corresponding closed-loop systems  in Theorems \ref{th20237141510}, \ref{th2023829953} and Corollary \ref{Co202396}. The current results can already demonstrate all the
%   contribution of controller
%  \end{remark}

 %
%  In Section \ref{NumSim}, we will make numerical simulation for system \dref{2023829945constant} to show the theoretical results in Theorem \ref{th2023829953} visually.

 \section{Proof of Theorem  \ref{Th20238281955} }\label{PfTh1}
 Before proving   Theorem  \ref{Th20238281955}, we first give two lemmas.
\begin{lemma}\label{Lm2023828938}
Suppose that  $b\neq 0$,  $c_1>0$ and $u_0\in L^2_{\rm loc}(0,\infty)$. Then, for any initial state
$(\tilde{w}(\cdot,0),\tilde{\zeta}(0))\in \H\times\R$ and for any $T>0$, system
  \dref{20237122040frac1b} admits a unique weak solution $(\tilde{w}(\cdot,t),\tilde{\zeta}(t))\in C(0,T;\H\times\R)$  such that
 \begin{equation}\label{2023828938}
  \int_0^t\int_{0}^{1}\tilde{w}_x^2(x,s)dx+c_1\tilde{w}^2(1,s)ds = F(0)-{F}(t)
 \end{equation}
 and
  \begin{equation}\label{2023828938E}
 \int_0^t \int_{0}^{1}\tilde{w}_{ x}^2(x,s)dx  +b u_0(s) \tilde{\zeta} (s)\tilde{w}(1,s)
    + c_1     \tilde{w} ^2(1,s)ds= E(0)-E(t),
 \end{equation}
 where  $0<t\leq T$  and
 \begin{equation}\label{20237141056}
 F(t)=E(t)+\frac{|b|}{2}\tilde{\zeta}^2(t),\ \ E(t)=\frac12\int_{0}^{1}\tilde{w}^2(x,t)dx .
\end{equation}
\end{lemma}
\begin{proof}
 We will prove Lemma \ref{Lm2023828938}  by Galerkin method \cite{Evans}. Let
 \begin{equation}\label{wxh201912303Ad429}
 \left\{\begin{array}{l}
 \disp \phi_n(x)=\sqrt{2}  \sin \sqrt{\lambda_n}x ,\crr
   \lambda_n=\left(n-\frac{1}{2}\right)^2\pi^2,
  \end{array}\right. \ \ x\in[0,1],\ \ n\geq 1.
 \end{equation}
Then,  $\{\phi_n(\cdot) \}_{n=1}^{\infty}$ forms an orthonormal   basis for  $\H$.
%and $\langle \phi_{jx},\phi_{kx}\rangle_{\H}=0$ if and only if $j\neq k$, $j,k=1,2,\cdots$.
%
% and forms an orthogonal basis for $H^1_L(0,1)=\{f\in H^1(0,1)\ | \ f(0)=0\}$.
%satisfies
%\begin{equation}\label{wxh201912302Ad429}
%\left\{\begin{array}{l}
%\phi_n''(x)=-\lambda_n\phi_n(x),\crr
% \phi_n(0)=\phi'_n(1)=0,
%\end{array}\right.n=1,2,\cdots.
%\end{equation}
 Fix a positive integer $N$, let $\mathcal{P}_N$ be the orthogonal projection of $\H$
 onto  $V_N:=\mbox{the linear span of }  \{\phi_j:j=1,2,\cdots,N  \} $.
Let
   \begin{equation} \label{2018961642}
    \tilde{w}_N(\cdot,t)=\sum_{j=1}^{N}\tilde{w}_{N,j}(t)\phi_j,\ \   \tilde{\zeta}_N(t)
    \end{equation}
     be
the approximate solutions that   satisfy the following system of ordinary differential equations:
 \begin{equation} \label{2023818918}
\left\{  \begin{array}{l}
\disp \langle \tilde{w}_{Nt}  (\cdot,t),\phi_j\rangle_{\H}
+ \langle \tilde{w}_{Nx}(\cdot,t)  , \phi_{jx}\rangle_{\H}
\disp =
   -\phi_j(1)      \left[   b  \tilde{\zeta}_N(t)u_0(t)
    + c_1     \tilde{w}_N(1,t) \right]  ,\crr
 \dot{\tilde{\zeta}}_N(t)=    \mbox{sgn}(b)u_0(t)  \tilde{w}_N (1,t),\crr
\disp \tilde{w}_N(\cdot,0)=\mathcal{P}_N\tilde{w}(\cdot,0), \ \ \tilde{\zeta}_N(0)=\tilde{\zeta}(0),\ \ \ j=1,2\cdots, N.
 \end{array}\right.
\end{equation}
 Clearly  the initial value  satisfies
   \begin{equation} \label{20189101501}
\lim\limits_{N\to+\infty} \|\tilde{w}_N(\cdot,0)-\tilde{w}(\cdot,0)\|_{\H} =0
\end{equation}
and system  \dref{2023818918} is an initial-value problem for a first-order $N+1$
system of linear ordinary differential equations with the  time-varying coefficient $u_0(t)$ and the
unknown functions $\tilde{w}_{N,j}, \tilde{\zeta}_N$.
By the theory of  ordinary differential equations,
  for every  $N \geq 1$ and any $T>0$, \dref{2023818918}  has a solution
$\tilde{w}_{N,j}, \tilde{\zeta}_N\in C^1[0,  T]$.

In terms of system \dref{2023818918}, we let
\begin{equation}\label{2023818807N}
 F_N(t)=E_N(t)+\frac{|b|}{2}\tilde{\zeta}_N^2(t),\ \ E_N(t)=\frac12\int_{0}^{1}\tilde{w}_N^2(x,t)dx .
\end{equation}
Then \dref{20189101501} implies  that
 \begin{equation}\label{20238291125}
 F_N(0)\to F(0)\ \ \mbox{as}\ \ N\to+\infty .
\end{equation}
%where
%\begin{equation}\label{20238291126}
% F(0)=E(0)+\frac{|b|}{2}\tilde{\zeta}^2(0),\ \ E(0)=\frac12\int_{0}^{1}\tilde{w}^2(x,0)dx .
%\end{equation}

Multiplying the first  equation in   \dref{2023818918} by  $\tilde{w}_{N,j}(t)$
 and summing  for $j = 1, \cdots ,N$, and
 multiplying the second   equation in   \dref{2023818918}
 by $\tilde{\zeta}_N(t)$, we get
  \begin{equation}\label{2023818806EN}
\dot{E}_N(t)= -\int_{0}^{1}\tilde{w}_{Nx}^2(x,t)dx -b u_0(t) \tilde{\zeta}_N(t) \tilde{w}_N (1,t)
    - c_1     \tilde{w}_N^2(1,t)
\end{equation}
  and
\begin{equation}\label{2023818940}
 \begin{array}{l}
\dot{F}_N(t)= \disp -\int_{0}^{1}\tilde{w}_{Nx}^2(x,t)dx-c_1\tilde{w}_N^2(1,t)
\leq 0.
\end{array}
\end{equation}
That is, for any $t\in [0,T]$,
 \begin{equation}\label{20238291501}
\int_0^t\int_{0}^{1}\tilde{w}_{Nx}^2(x,s)dx +b u_0(s) \tilde{\zeta}_N(s)  \tilde{w}_N (1,s)
    + c_1     \tilde{w}_N^2(1,s)ds=E_N(0)-E_N(t)
\end{equation}
and
\begin{equation}\label{2023828938N}
  \int_0^t\int_{0}^{1}\tilde{w}_{Nx}^2(x,s)dx+c_1\tilde{w}^2_N(1,s)ds  +{F}_N(t)= F_N(0).
 \end{equation}
% By \dref{20238291125}, there exists a positive constant $C$
%  that is independent of $N$     such that $ F_N(0)<C$. Consequently, for any given $T>0$, the energy estimate
%  \dref{2023828938N} implies that the sequence  $\{\tilde{w}_N \}$
%  is  bounded in    $L^{2}(0,T;\H)$ and  $L^2(0,T;H ^1(0,1))$,
%      $\{\tilde{\zeta}_N(t)\}$ is  bounded
%   in $L^{2}(0,T) $
%  and  $\{\tilde{w} _N(1,t)\}$
%is  bounded
%   in   $L^2(0,T) $.

Let $(\tilde{w}_N(\cdot,t),\tilde{\zeta}_{N}(t)) ,   (\tilde{w}_L(\cdot,t),\tilde{\zeta}_{L}(t))$  be two   approximate solutions, and without loss of generality, we assume $N > L$. Set
$$
\left\{\begin{array}{l}
\disp \tilde{w}_{NL}(\cdot,t) =\tilde{w}_N(\cdot,t)-\tilde{w}_L(\cdot,t)=
\sum_{j=1}^N[\tilde{w}_{N,j}(t)-\tilde{w}_{L,j}(t)]\phi_j,\crr
\disp   \tilde{\zeta}_{NL}(t)=
\tilde{\zeta}_{N}(t)-\tilde{\zeta}_{L}(t),
\end{array}\right.
$$
 with the understanding that
$\tilde{w}_{L,j}(t)\equiv 0$ when $j>L$. Then $\tilde{w}_{NL}(\cdot,t)$  and $\tilde{\zeta}_{NL}(t)$ satisfy
\begin{equation} \label{20239152217}
\left\{  \begin{array}{l}
\disp \langle \tilde{w}_{NLt}  (\cdot,t),\phi_j\rangle_{\H}
+ \langle \tilde{w}_{NLx}(\cdot,t)  , \phi_{jx}\rangle_{\H}
\disp =
  - \phi_j(1)      \left[  bu_0(t) \tilde{\zeta}_{NL}(t)
    + c_1    \tilde{w}_{NL}(1,t) \right]  ,\crr
  \dot{\tilde{\zeta}}_{NL}(t)=    \mbox{sgn}(b)u_0(t)  \tilde{w}_{NL} (1,t), \crr
\disp \tilde{w}_{NL}(\cdot,0)=\mathcal{P}_N\tilde{w}(\cdot,0)-\mathcal{P}_L\tilde{w}(\cdot,0), \
\tilde{\zeta}_{NL}(0)=0,\ \ j=1,2\cdots ,N.
 \end{array}\right.
\end{equation}
Let
\begin{equation} \label{20239152218}
 F_{NL}(t)=\frac12\int_{0}^{1}\tilde{w}_{NL}^2(x,t)dx +\frac{|b|}{2}\tilde{\zeta}_{NL}^2(t),\ \ t\geq0.
\end{equation}
 Since system \dref{2023818918} takes the same form as system  \dref{20239152217},
 we
 get
   \begin{equation}\label{2023828938NAd915}
  \int_0^t\int_{0}^{1}\tilde{w}_{NLx}^2(x,s)dx+c_1\tilde{w}^2_{NL}(1,s)ds  +{F}_{NL}(t)= F_{NL}(0) ,
 \end{equation}
 as we have obtained   \dref{2023828938N}.
 Note that
   \begin{equation}\label{20239152220}
  F_{NL}(0) \to 0\ \mbox{as}\ N,L\to+\infty,
\end{equation}
\dref{2023828938NAd915} implies that, for any given $t\in[0,T]$,
   $\{\tilde{w}_N(\cdot,t) \}$
  is  a Cauchy sequence in    $ \H$ and
    $\{\tilde{\zeta}_N(t)\}$ is  a Cauchy sequence
   in $\R $. Moreover,
  $\{\tilde{w}_N \}$
  is  a Cauchy sequence in
   $L^2(0,T;H ^1(0,1))$
  %    $\{\tilde{\zeta}_N(t)\}$ is  a Cauchy sequence
 %  in $L^{2}(0,T) $,
  and  $\{\tilde{w} _N(1,t)\}$
is  a Cauchy sequence
   in   $L^2(0,T) $.
Hence, there exist  $(\tilde{w}(\cdot,t),\tilde{\zeta}(t))\in C(0,T;\H\times\R)$ such that
  \begin{equation} \label{2023951453}
  \left\{\begin{array}{l}
  \tilde{w}_{N}  (\cdot,t)\to \tilde{w}(\cdot,t)\ \  \mbox{  in}\ \ \H,\ \ \forall\ t\in[0,T],
   \crr
    \disp    \tilde{\zeta}_N(t)\to \tilde{\zeta}(t) \ \  \mbox{  in}\ \ \R,\ \ \forall\ t\in[0,T],\crr
     \disp \tilde{w}_N \to \tilde{w} \ \ \mbox{  in}\ \   \ L^{2}(0,  T;H ^1(0,1)),\crr
    \disp \tilde{w}_N(1,\cdot)\to \tilde{w}(1,\cdot) \ \ \mbox{   in}\ \   \ L^{2}( 0, T ),
   \end{array}\right.\mbox{as}\ \ N\to+\infty.
\end{equation}

For any  $T>0$ and $0<t<T$,  by integrating  \dref{2023818918} from $0$ to $t$, we obtain
 \begin{equation} \label{2023818918inte}
\left\{  \begin{array}{l}
\disp \langle \tilde{w}_{N}  (\cdot,t),\phi_j\rangle_{\H}
+ \int_0^t\langle \tilde{w}_{Nx}(\cdot,\tau)  , \phi_{jx}\rangle_{\H}d\tau
\disp =\langle \tilde{w}_{N}  (\cdot,0),\phi_j\rangle_{\H}\crr
\disp  - \phi_j(1) \int_0^t     \left[   b  \tilde{\zeta}_N(\tau)u_0(\tau)
    + c_1     \tilde{w}_N(1,\tau)     \right]d\tau  ,\crr
  \tilde{\zeta} _N(t)=\disp \tilde{\zeta} (0)  +\mbox{sgn}(b) \int_0^t \tilde{w}_N (1,\tau)u_0(\tau)d\tau,
 %\disp \tilde{w}_N(\cdot,0)=\mathcal{P}_N\tilde{w}(\cdot,0), \ \ %\tilde{\zeta}_N(0)=\tilde{\zeta}(0),\
 \end{array}\right. \ \ j=1,2\cdots ,N.
\end{equation}
 We  pass to the limit  as  $N\to+\infty$ in \dref{2023818918inte} to get
 \begin{equation} \label{2023818918inteweak}
\left\{  \begin{array}{l}
\disp \langle \tilde{w}   (\cdot,t),\phi_j\rangle_{\H}
+ \int_0^t\langle \tilde{w}_{x}(\cdot,\tau)  , \phi_{jx}\rangle_{\H}d\tau
\disp =\langle \tilde{w}   (\cdot,0),\phi_j\rangle_{\H}\crr
\disp  - \phi_j(1) \int_0^t     \left[   b  \tilde{\zeta} (\tau)u_0(\tau)
    + c_1     \tilde{w} (1,\tau)     \right]d\tau  ,\crr
  \tilde{\zeta}  (t)=\disp \tilde{\zeta} (0)  +\mbox{sgn}(b) \int_0^t \tilde{w} (1,\tau)u_0(\tau)d\tau,
 %\disp \tilde{w}_N(\cdot,0)=\mathcal{P}_N\tilde{w}(\cdot,0), \ \ \ \ \ j=1,2\cdots ,N.
 \end{array}\right.
\end{equation}
which implies that $ (\tilde{w},\tilde{\zeta} )  $ is a weak solution of the system \dref{20237122040frac1b}.   Since system  \dref{20237122040frac1b}  is a linear system, the uniqueness of the solution is trivial.
   % By dominated convergence theorem and taking \dref{20238291125} into account,
   By virtue of \dref{2023951453},
     we let $N\to+\infty$ in \dref{2023828938N} and \dref{20238291501} to obtain  \dref{2023828938} and  \dref{2023828938E}, respectively.
\end{proof}

\begin{lemma}\label{Lm2023811150}
Let  $b\neq 0$ and $c_1>0$. Suppose that the controller
  $u_0\in W $.
  Then,
    for any initial state
$(\tilde{w}(\cdot,0),\tilde{\zeta}(0))\in \H\times\R$, system
  \dref{20237122040frac1b} admits a unique weak solution $(\tilde{w}(\cdot,t),\tilde{\zeta}(t))\in C(0,\infty;\H\times\R)$  such that
   \begin{equation}\label{20238291544}
   \int_{0}^{\infty}\tilde{w}^2(0,t)+\tilde{w}^2(1,t)dt<+\infty,
 \end{equation}
 \begin{equation}\label{2023811157}
 \|\tilde{w}(\cdot,t)\|_{\H}\to 0 \ \ \mbox{as}\ \ t\to+\infty
 \end{equation}
and
\begin{equation}\label{2023811157B}
  \tilde{\zeta} (t)  \to \zeta_* \ \ \mbox{as}\ \ t\to+\infty ,
 \end{equation}
where $\zeta_*$ is a constant that may not be zero.
If we suppose further  that there exists a time $\tau>0$ such that the  controller $u_0 $  satisfies
\dref{202381809*},
%\begin{equation}\label{202381809}
%      \int_{t}^{t+T}
%       u_0^2(s) ds\geq p_T , \ \ \forall\ t\geq0,\ T>0,
%\end{equation}
% \begin{equation}\label{202381809*}
%      \lim_{t\to+\infty}\int_{t}^{t+T}
%       u_0 (s) ds \neq0,
%\end{equation}
then,
\begin{equation}\label{20238231156}
  \tilde{\zeta} (t)  \to  0 \ \ \mbox{as}\ \ t\to+\infty .
 \end{equation}
\end{lemma}
\begin{proof}
   The existence of the solution can be obtained by Lemma \ref{Lm2023828938}.
   It is sufficient to prove \dref{20238291544},  \dref{2023811157} and \dref{2023811157B}.
     We first prove \dref{20238291544} and the convergence \dref{2023811157}.
By \dref{2023828938} and \dref{20237141056}, we have
%  \begin{equation}\label{20237191102}
%%\dot{F}(t)=-\int_{0}^{1}\tilde{w}_x^2(x,t)dx-c_1\tilde{w}^2(1,t)\leq0,
%\tilde{\zeta}^2(t)\leq \frac{2}{|b|} F(t) \leq  \frac{2}{|b|}F(0),
%\end{equation}
 \begin{equation}\label{2023720757}
\tilde{\zeta}^2(t)\leq \frac{2}{|b|} F(t) \leq  \frac{2}{|b|}F(0).
\end{equation}
By \cite[p.258, Theorem 1]{Evans}, there exists a positive constant $M$ such that
 \begin{equation}\label{202411242219FD}
\tilde{w}^2(0,t) \leq M \left [\tilde{w}^2(1,t)+\disp \int_{0}^{1}\tilde{w}_x^2(x,t)dx \right].
\end{equation}
Passing to the limit  as $ t\to+\infty$ in \dref{2023828938}, we obtain
 \begin{equation}\label{20238231205}
 \int_0^{\infty}\left[ \int_{0}^{1}\tilde{w}_x^2(x,t)dx+c_1\tilde{w}^2(1,t)\right] dt
  \leq F(0),
 \end{equation}
 which, together with \dref{202411242219FD}, leads to  \dref{20238291544}.

By \dref{2023828938E}, \dref{2023720757} and the Poincar\'{e}'s inequality,  there exists a positive constant $\omega$ such that
   \begin{equation*}\label{2023941442}
    \begin{array}{rl}
 \disp \dot{E}(t)&\disp =-\int_{0}^{1}\tilde{w}_{ x}^2(x,t)dx  -b u_0(t) \tilde{\zeta} (t)\tilde{w}(1,t)
    - c_1     \tilde{w} ^2(1,t) \crr
    &\leq -\omega E(t)  +\sqrt{2|b|F(0)} \cdot |u_0(t)|    |  \tilde{w}(1,t)|
         .
        \end{array}
 \end{equation*}
Applying  Lemma \ref{Lm20231004} to this inequality, we get
   \dref{2023811157} easily.

Now we prove the convergence \dref{2023811157B}. Owing to \dref{2023828938}  and \dref{20237141056}, we conclude that
 \begin{equation}\label{202395808}
  \dot{F}(t)\leq0\ \ \mbox{and}\ \ 0\leq F(t)\leq F(0),\ \ \forall\ t\geq0.
 \end{equation}
By the monotone convergence theorem, there exists a $F_*\geq0$  such that
$F(t)\to F_*$ as $t\to+\infty$. Furthermore,
\begin{equation}\label{202395815}
   \lim_{t\to+\infty}\tilde{\zeta}^2(t)=\frac{2}{|b|}  \lim_{t\to+\infty}\left[F(t)-E(t)\right]=
   \frac{2 F_*}{|b|}
 \end{equation}
and hence
\begin{equation}\label{202395815abs}
   \lim_{t\to+\infty}|\tilde{\zeta} (t)|=
   \sqrt{\frac{2 F_* }{|b|} } .
 \end{equation}

When $F_*=0$, \dref{202395815abs} implies that $\tilde{\zeta} (t)\to0$ as $t\to+\infty$. It suffices to consider the case that $F_*\neq0$.  By \dref{202395815}, there exists a $t_0>0$ such that
\begin{equation}\label{202395821}
 \tilde{\zeta}^2(t)>
   \frac{F_*}{|b|}>0,\ \ \forall\ t>t_0.
 \end{equation}
Hence, $\tilde{\zeta}(t)\neq 0$ for all $t>t_0$.
Note that $\tilde{\zeta}\in C[0,\infty)$, it follows from the intermediate value theorem
that
\begin{equation}\label{202395827}
 |\tilde{\zeta}(t)|=\tilde{\zeta}(t) \ \ \mbox{or}\ \  |\tilde{\zeta}(t)|=-\tilde{\zeta}(t),\ \ \forall\ t>t_0,
 \end{equation}
which, together with \dref{202395815abs}, leads to
  \begin{equation*}\label{2023831906}
  \lim_{t\to+\infty} \tilde{\zeta} (t) =
   \sqrt{\frac{2 F_* }{|b|} } \ \ \mbox{or}\ \    \lim_{t\to+\infty} \tilde{\zeta} (t) =
   -\sqrt{\frac{2 F_* }{|b|} } .
 \end{equation*}
So we have proved  the convergence \dref{2023811157B}.

%By \dref{20237141056} and \dref{2023828938}, the function $F$ is   a monotonically decreasing function with a lower bound $0$. By the monotone bounded convergence theorem,
%there exists a constant $F_0$ such that  $\lim_{t\to\infty}F(t)=F_*$.
%Hence,
%$$
%\tilde{\zeta}^2(t)=\frac{2}{|b|}F(t)-\frac{1}{|b|}E(t)\to\frac{2}{|b|}F_*\ \ \mbox{as}\ \ t\to+\infty,
%$$
%which leads to \dref{2023811157B} with $\zeta_*=\frac{2}{|b|}F_*$.

In order to prove \dref{20238231156},
we suppose   that the condition   \dref{202381809*}  holds but $\zeta_*\neq0$.
%For any $\tau>0$, a simple computation shows that
%\begin{equation}\label{20238291530}
%   \zeta_*\int_{t}^{t+\tau} u_0(s)ds=
%  \int_{t}^{t+\tau}[\zeta_*-\tilde{\zeta}^2(s)]u_0(s)ds +
%   \int_{t}^{t+\tau} \tilde{\zeta}^2(s) u_0(s)ds
% \end{equation}
 Define
\begin{equation}\label{20238251127}
\rho (t)=\int_0^1x\tilde{w} (x,t)dx,\ \ t\geq0 .
%\ \ \mbox{and}\ \
%\vartheta_n(t)=-c_1\tilde{w} (1,t)\phi_n(1)
% +\phi_{nx}(0)\tilde{w}(0,t).
\end{equation}
Finding the derivative along the solution of system   \dref{20237122040frac1b} to get
\begin{equation}\label{202382511250}
\begin{array}{l}
 \disp  \dot{\rho} (t)=     -b\tilde{\zeta} (t)u_0(t)
 -(1+c_1)\tilde{w} (1,t)
 + \tilde{w}(0,t).
 \end{array}
 \end{equation}
%Consequently,
%\begin{equation}\label{20238251125}
%\begin{array}{l}
% \disp \tilde{\zeta}(t)\dot{\rho}_n(t)=  -\lambda_n\rho_n(t) \tilde{\zeta}(t) -b\tilde{\zeta}^2(t)u_0(t)\phi_n(1)+\tilde{\zeta}(t)\vartheta_n(t)
% \end{array}
% \end{equation}
 % \begin{equation}\label{wxh201912303Ad429}
% \left\{\begin{array}{l}
% \disp \phi_n(x)=\sqrt{2}  \sin \sqrt{\lambda_n}x ,\crr
%   \lambda_n=\left(n-\frac{1}{2}\right)^2\pi^2,
%  \end{array}\right. \ \ x\in[0,1],\ \ n\geq 1.
% \end{equation}
%Then,  $\{\phi_n(\cdot) \}_{n=1}^{\infty}$ forms an orthonormal   basis for  $L^2[0,1]$ and
%satisfies
%\begin{equation}\label{wxh201912302Ad429825}
%\left\{\begin{array}{l}
%\phi_n''(x)=-\lambda_n\phi_n(x),\crr
% \phi_n(0)=\phi'_n(1)=0,
%\end{array}\right.n=1,2,\cdots.
%\end{equation}
% \begin{equation}\label{20238251116}
%\begin{array}{l}
% \disp \langle \tilde{w}_t(\cdot,t),\phi_n\rangle_{\H}= \langle \tilde{w}_{xx}(\cdot,t),\phi_{n}\rangle_{\H}=
% \tilde{w}_x(1,t)\phi_n(1)+\phi_{nx}(0)\tilde{w}(0,t)-\lambda_n\langle \tilde{w} (\cdot,t),\phi_{n}\rangle_{\H}\crr
% \disp=  -b\tilde{\zeta}(t)u_0(t)\phi_n(1)-c_1\tilde{w} (1,t)\phi_n(1)
% +\phi_{nx}(0)\tilde{w}(0,t)-\lambda_n\langle \tilde{w} (\cdot,t),\phi_{n}\rangle_{\H}\crr
% \end{array}
% \end{equation}
% In view the last equation of system   \dref{20237122040frac1b},  for any $\tau>0$,  we use
% integration by parts
 % to get
 For any $t>0$ and   $\tau>0$, it follows from \dref{2023811157} and  \dref{20238251127} that
  \begin{equation}\label{20238251130toinfty}
 \int_{t}^{t+\tau} \dot{\rho} (s)ds=
    {\rho} (t+\tau)-   {\rho} (t) \to 0\ \ \mbox{as}\  \ t\to+\infty.
 \end{equation}
% which, together with \dref{2023811157}, \dref{20238251127} and \dref{2023720757}, leads to
% \begin{equation}\label{20238251130toinfty}
% \int_{t}^{t+\tau} \tilde{\zeta}(s)\dot{\rho}_n(s)ds \to 0\ \ \mbox{as}\  \ t\to+\infty.
% \end{equation}
 %By the Sobolev trace-embedding, there exists a positive constant $M $ such that
%   \begin{equation}\label{20238291540}
%   \tilde{w}^2(0,t)\leq M \left[ \tilde{w}^2(1,t)+\int_{0}^{1}\tilde{w}_x^2(x,t)dx\right]
% ,
% \end{equation}
% which, together with \dref{2023828938}, leads to
%   \begin{equation}\label{20238291544}
%   \int_{0}^{\infty}\tilde{w}^2(0,t)+\tilde{w}^2(1,t)dt<+\infty.
% \end{equation}
By \dref{20238291544}, for any $\tau>0$, we have
 \begin{equation}\label{20238291546}
   \int_{t}^{t+\tau}\tilde{w}^2(0,s)+\tilde{w}^2(1,s)ds\to 0 \ \ \mbox{as}\  \ t\to+\infty,
 \end{equation}
which yields
   \begin{equation}\label{20238251135}
 \int_{t}^{t+\tau} \tilde{w}(0,s) -(1+c_1)\tilde{w} (1,s)
  ds
  \to0\ \ \mbox{as}\  \ t\to+\infty.
 \end{equation}
 Combining \dref{202382511250}, \dref{20238251130toinfty} and \dref{20238251135}, we arrive at
 \begin{equation}\label{20238291552}
 b  \int_{t}^{t+\tau}\tilde{\zeta} (s) u_0(s)ds \to0\ \ \mbox{as}\  \ t\to+\infty.
 \end{equation}

 On the other hand,
 it follows from  \dref{2023811157B} that
\begin{equation}\label{20238291555}
 b  \int_{t}^{t+\tau}[\zeta_*-\tilde{\zeta} (s)]u_0(s)ds \to0\ \ \mbox{as}\  \ t\to+\infty.
  \end{equation}
%it follows from  \dref{202382511250} and \dref{20238251135} that
Note that
\begin{equation}\label{20238251136}
 b  \int_{t}^{t+\tau}\zeta_* u_0(s)ds=
 b  \int_{t}^{t+\tau}[\zeta_*-\tilde{\zeta} (s)]u_0(s)ds +
 b  \int_{t}^{t+\tau} \tilde{\zeta} (s) u_0(s)ds.
  \end{equation}
Combining \dref{20238291555}, \dref{20238251136} and \dref{20238291552}, we have
\begin{equation}\label{20238291557}
 b \zeta_* \int_{t}^{t+\tau} u_0(s)ds \to0\ \ \mbox{as}\  \ t\to+\infty,
  \end{equation}
 which, together with the assumption  $ b \zeta_* \neq0$,    contradicts  to the   condition \dref{202381809*}.
 \end{proof}

{\it Proof of Theorem \ref{Th20238281955}.}
By Lemma \ref{Lm2023828938}, we can suppose that  $(\tilde{w}(\cdot,t),\tilde{\zeta}(t))\in C(0,\infty;\H\times\R)$  is a weak solution to  system
  \dref{20237122040frac1b}.
  Let
   \begin{equation}\label{20238291644}
   \zeta(t)=\frac1b-\tilde{\zeta}(t),\ \ t\geq0.
   \end{equation}
       Then   $\zeta\in C[0,\infty)$ and
       $u(t)=\zeta(t)u_0(t)$ belongs to $L^2_{\rm loc}(0,\infty)$.
    Hence,
system \dref{20237122035} with $u(t)=\zeta(t)u_0(t)$ admits a unique solution
$ {w}(\cdot,t) \in C(0,\infty;\H)$.  Let
 \begin{equation}\label{20238291646}
   \hat{w}(x,t)=w(x,t)-\tilde{w}(x,t),\ \ x\in[0,1],\ t\geq0.
   \end{equation}
A simple computation   shows that   $(\hat{w}(\cdot,t),\zeta(t))\in C(0,\infty;\H\times\R)$ defined by \dref{20238291644}
 and \dref{20238291646} is a weak solution of  the observer \dref{20237122036}.
Furthermore, the convergence    \dref{20238282000}  and \dref{20238282005} can be obtained by Lemma \ref{Lm2023811150} directly.
\hfill$\Box$

%We first consider the following transformed system
% \begin{equation}\label{20238291624}
%\left\{\begin{array}{l}
% \hat{w}_{t}(x,t)=\hat{w}_{xx}(x,t)   , \; x\in (0,1),
% \crr
% \hat{w}_{x} (0,t)=-q \hat{w}(0,t)-q\tilde{w}(0,t),\  \
%\hat{w}_{x}(1,t)= u_0(t)+c_1  \tilde{w}(1,t)  ,   \crr
%  \tilde{w}_{t}(x,t)=\tilde{w} _{xx}(x,t)   ,
% \crr
% \tilde{w} _{x} (0,t)= 0 ,\ \ \tilde{w} _{x}(1,t)=-b\tilde{\zeta}(t)u_0(t) -c_1\tilde{w} (1,t),    \crr
%\dot{\tilde{\zeta}}(t)=  \mbox{sgn}(b) u_0(t)\tilde{w} (1,t),
% \end{array}\right.
%\end{equation}
%where $u_0\in L^{\infty}(0,\infty)$.

 \section{Proof of Theorem  \ref{th20237141510}}\label{PfTh2}
 Before proving Theorem  \ref{th20237141510}, we first consider the following transformed system:
 \begin{equation}\label{20238292127}
\left\{\begin{array}{l}
 \tilde{w}_{t}(x,t)=\tilde{w} _{xx}(x,t)   ,
 \crr
 \tilde{w} _{x} (0,t)= 0 ,\
  \disp\tilde{w} _{x}(1,t)= b(q+c_0)\tilde{\zeta}(t)  \check{w}(1,t)-c_1\tilde{w} (1,t),    \crr
\disp\dot{\tilde{\zeta}}(t)=  - (q+c_0) \mbox{sgn}(b)  \tilde{w} (1,t)\check{w}(1,t) ,\crr
  \check{w}_{t}(x,t)=\check{w}_{xx}(x,t)  +q^2e^{qx}\tilde{w}(0,t) ,\crr
 \check{w}_{x} (0,t)=-q\tilde{w}(0,t),\ \
\disp  \check{w}_{x}(1,t)= - c_0  \check{w}(1,t) +c_1  \tilde{w}(1,t)   .
 \end{array}\right.
\end{equation}
%We consider system \dref{20238292127} in the Hilbert space $\X=\H\times \R\times \R$.
%Define the operator $\mathcal{A}: D(\mathcal{A})\subset\X\to\X$ by
%\begin{equation}\label{20239192102}
%\left\{\begin{array}{l}
%\mathcal{A}(f,h,g)=(f'',0,g''),\ \ \forall\ (f,h,g)\in D(\mathcal{A}),\crr
%\disp D(\mathcal{A})=\Big{\{}(f,h,g)\in H^2(0,1)\times\R\times H^2(0,1)\ |\crr
% f'(0)=0,f'(1)=-c_1f(1),g'(0)=-qf(0),g'(1)=-c_0g(1)+c_1f(1)\Big{\}}.
%\end{array}\right.
%\end{equation}
%Define the operator $\mathcal{F}: D(\mathcal{A})\subset\X\to\X$ by
%\begin{equation}\label{20239192103}
%\left\{\begin{array}{l}
%\mathcal{F}(f,h,g)=(f'',0,g''),\ \ \forall\ (f,h,g)\in D(\mathcal{A}),\crr
%\disp D(\mathcal{A})=\Big{\{}(f,h,g)\in H^2(0,1)\times\R\times H^2(0,1)\ |\crr
% f'(0)=0,f'(1)=-c_1f(1),g'(0)=-qf(0),g'(1)=-c_0g(1)+c_1f(1)\Big{\}}.
%\end{array}\right.
%\end{equation}

\begin{lemma}\label{Lm2023828938Ad92}
Suppose that  $b\neq 0$,  $c_0 >0$ and $c_1>0$. Then for any initial state
$(\tilde{w}(\cdot,0),\check{w}(\cdot,0),\tilde{\zeta}(0))\in \H^2\times\R$,
system
  \dref{20238292127} admits a   weak solution $(\tilde{w}(\cdot,t),\check{w}(\cdot,t),\tilde{\zeta}(t))\in C(0,\infty;\H^2\times\R)$ such that

 \begin{equation}\label{2023951709}
   \int_{0}^{\infty} \int_{0}^{1}\tilde{w}_{x}^2(x,t)+\check{w}_{x}^2(x,t)dx +         \tilde{w}^2(1,t)+\check{w}^2(1,t)dt<+\infty.
  % C_0\|(\tilde{w}(\cdot,0),\check{w}(\cdot,0), \tilde{\zeta}(0))\|_{\H^2\times \R}^2,
 \end{equation}
% where   $C_0 $ is a positive constant.
  \end{lemma}
\begin{proof}
 We will prove Lemma \ref{Lm2023828938Ad92}  by Galerkin method. Let
 $\{\phi_n \} $ be given by \dref{wxh201912303Ad429}.
 %Then $\phi_j(0)=0$, $j=1,2,\cdots$.
Suppose that
   \begin{equation} \label{2018961642Ad92}
    \tilde{w}_N(\cdot,t)=\sum_{j=1}^{N}\tilde{w}_{N,j}(t)\phi_j,\ \
    \check{w}_N(\cdot,t)=\sum_{j=1}^{N}\check{w}_{N,j}(t)\phi_j,\ \   \tilde{\zeta}_N(t)
    \end{equation}
   satisfy the following system of ordinary differential equations:
 \begin{equation} \label{2023818918Ad92}
\left\{  \begin{array}{l}
\disp \langle \tilde{w}_{Nt}  (\cdot,t),\phi_j\rangle_{\H}
+ \langle \tilde{w}_{Nx}(\cdot,t)  , \phi_{jx}\rangle_{\H}
\disp =
    \phi_j(1)      \left[   b (q+c_0) \tilde{\zeta}_N(t)   \check{w}_N(1,t)
    - c_1     \tilde{w}_N(1,t) \right]  ,\crr
    \disp \langle \check{w}_{Nt}  (\cdot,t),\phi_j\rangle_{\H}
+ \langle \check{w}_{Nx}(\cdot,t)  , \phi_{jx}\rangle_{\H}
\disp = \phi_j(1)      \left[
     c_1     \tilde{w}_N(1,t) -c_0\check{w}_N(1,t)\right]  ,\crr
 \dot{\tilde{\zeta}}_N(t)=    -\mbox{sgn}(b)(q+c_0)   \tilde{w}_N (1,t) \check{w}_N(1,t),\crr
\disp \tilde{w}_N(\cdot,0)=\mathcal{P}_N\tilde{w}(\cdot,0), \  \disp \check{w}_N(\cdot,0)=\mathcal{P}_N\check{w}(\cdot,0), \  \tilde{\zeta}_N(0)=\tilde{\zeta}(0),\ \   j=1,2\cdots ,N,
 \end{array}\right.
\end{equation}
where   $\mathcal{P}_N$ be the orthogonal projection of $\H$
 onto  $V_N:=\mbox{the linear span of }  \{\phi_j:j=1,2,\cdots,N  \} $.
 Clearly  the initial values satisfy
   \begin{equation} \label{20189101501Ad92}
\lim\limits_{N\to+\infty} \|\tilde{w}_N(\cdot,0)-\tilde{w}(\cdot,0)\|_{\H}+\|\check{w}_N(\cdot,0)-\check{w}(\cdot,0)\|_{\H}
 =0.
\end{equation}
%and
%   \begin{equation} \label{20189101501Ad92Check}
%\lim\limits_{N\to+\infty}  \|\check{w}_N(\cdot,0)-\check{w}(\cdot,0)\|_{\H} =0.
%\end{equation}
%and system  \dref{2023818918} is an initial-value problem for a first-order $2N+1$
%system of ordinary differential equations with the
%unknown functions $\tilde{w}_{N,j},\check{w}_{N,j},  \tilde{\zeta}_N$.
It follows from the
Cauchy-Peano theorem that for every  $N \geq 1$, \dref{2023818918Ad92}  has a solution
$\tilde{w}_{N,j}, \check{w}_{N,j}, \tilde{\zeta}_N\in C^1[0, T_N]$ for some $T_N > 0$.
 %Since system \dref{2023818918Ad92} is nonlinear, it is more difficult than system \dref{2023818918}
 % in Lemma \ref{Lm2023828938}.  So we divide the left proof into four  parts for clarity.

To address the nonlinearities  in \dref{2023818918Ad92}, we introduce
 \begin{equation}\label{202310131547}
 z_{N}(x,t)=\gamma \tilde{w}_{N}(x,t),\ \ \xi_{N}(t)=\gamma^2\tilde{\zeta}_{N}(t),
  %\ x\in[0,1],\  t\geq0,
\end{equation}
where $\gamma$ is a positive constant that is sufficiently large. Then, system \dref{2023818918Ad92} becomes
 \begin{equation} \label{202310131549}
\left\{  \begin{array}{l}
\disp \langle  z _{Nt}  (\cdot,t),\phi_j\rangle_{\H}
+ \langle  z _{Nx}(\cdot,t)  , \phi_{jx}\rangle_{\H}
\disp =
    \phi_j(1)      \left[   \frac{b (q+c_0)}{\gamma} \xi _N(t)   \check{w}_N(1,t)
    - c_1      z _N(1,t) \right]  ,\crr
    \disp \langle \check{w}_{Nt}  (\cdot,t),\phi_j\rangle_{\H}
+ \langle \check{w}_{Nx}(\cdot,t)  , \phi_{jx}\rangle_{\H}
\disp = \phi_j(1)      \left[
     \frac{c_1}{\gamma}      z _N(1,t) -c_0\check{w}_N(1,t)\right]  ,\crr
 \dot{\xi }_N(t)=    -\mbox{sgn}(b)(q+c_0)\gamma    z _N (1,t) \check{w}_N(1,t),\crr
\disp  z _N(\cdot,0)=\gamma\mathcal{P}_N \tilde{w} (\cdot,0), \  \disp \check{w}_N(\cdot,0)=\mathcal{P}_N\check{w}(\cdot,0), \  \xi _N(0)=\gamma^2\tilde{\zeta} (0),\ \   j=1,2\cdots ,N.
 \end{array}\right.
\end{equation}
We divide the left proof into four  parts for clarity.

{\it (1),  A priori estimation.}
%Let
%$\varepsilon$ and $\beta$ be two
%   positive constants that satisfy
%\begin{equation}\label{2023931456}
% 0<\varepsilon<\frac{4c_0}{c_1} \ \ \mbox{and}\ \ \frac{\varepsilon}{2}<\beta<\frac{2 c_0}{c_1}.
% \end{equation}
%  By Young's inequality, we have
% \begin{equation}\label{2023828938NAd92}
% | \check{w}_N(1,t)\tilde{w}_N(1,t)|\leq  \frac{ \beta}{2} \check{w}_N^2(1,t)+
% \frac{1}{2\beta }\tilde{w}_N^2(1,t).
% \end{equation}
Define
\begin{equation}\label{2023931622}
 L_N(t)=\frac12\int_{0}^{1}z_N^2(x,t)+\check{w}_N^2(x,t)dx
 +\frac{|b|}{2\gamma^2}\xi_N^2(t).
 \end{equation}
Multiplying the  first two equations  of  \dref{202310131549} by
$\gamma\tilde{w}_{N,j}(t)$,
 $\check{w}_{N,j}(t)$ respectively
  and summing  for $j = 1, \cdots ,N$,
  and
  multiplying the third    equation of    \dref{202310131549}
 by $\xi_N(t)$, we get
\begin{equation}\label{2023818940Ad93}
 \begin{array}{l}
\dot{L}_N(t)= \disp -\int_{0}^{1} z _{Nx}^2(x,t)+ \check{w} _{Nx}^2(x,t)dx-c_1 z _N^2(1,t)-c_0\check{w}_N^2(1,t)+\frac{c_1}{\gamma}z_N (1,t)\check{w}_N (1,t).
\end{array}
\end{equation}
By Cauchy's  inequality, we get
 \begin{equation}\label{2023931658}
 \begin{array}{l}
\dot{L}_N(t)\leq \disp   -\int_{0}^{1} z _{Nx}^2(x,t)+\check{w}_{Nx}^2(x,t)dx
-\left(c_1-\frac{c_1}{2\gamma}\right) z _N^2(1,t)-
\left(c_0-\frac{c_1}{2\gamma}\right)\check{w}_N^2(1,t)
\leq 0.
\end{array}
\end{equation}
If we choose $\gamma$ large enough such that
 \begin{equation}\label{2023100161043}
  \gamma>\max\left\{\frac12,\frac{c_1}{2c_0}\right\},
  \end{equation}
then, for any $t\geq0$,  there exists a positive constant $\alpha$  such that
\begin{equation}\label{2023951539}
 \begin{array}{l}
\disp  \alpha \int_0^t\int_{0}^{1} z _{Nx}^2(x,s)+\check{w}_{Nx}^2(x,s)dx+  z _N^2(1,s)+
 \check{w}_N^2(1,s)ds+ {L}_N(t) \leq L_N(0)< M_0,
\end{array}
\end{equation}
where
  $M_0$ is a positive constant  that is independent of $\gamma$ and $N$.
Combing \dref{2023951539},
 \dref{2023931622} and \dref{202310131547},  we can conclude that    the solution of system \dref{2023818918Ad92}
 can be extended to $\tilde{w}_{N,j},\check{w}_{N,j}, \tilde{\zeta}_N\in C^1[0, \infty)$.
 %   \begin{equation}\label{20238181107Ad92}
%  0\leq L_N(t)\leq L_N(0)<+\infty,\ \ \forall\ t\geq0,\  N=1,2,\cdots ,
%  \end{equation}
%and
%\begin{equation}\label{20238181614}
%\int_{0}^{\infty}\int_{0}^{1}\tilde{w}_{Nx}^2(x,s)dxds+     c_1\int_{0}^{\infty}\tilde{w}_N^2(1,t)dt\leq F(0) .
%  \end{equation}
 % Moreover,
%for any $t\geq0$ and any  positive integer $N$.

It follows from \dref{2023931622} and  \dref{2023951539} that
\begin{equation}\label{202310131615}
 |\xi_N(t)|<\sqrt{\frac{2 M_{0}}{|b|}}\gamma,\ \ \forall\ t\geq0.
\end{equation}
Moreover, it follows from the \dref{202310131549}   that
\begin{equation}\label{202310231510}
 \gamma(q+c_0)\left|\int_{0}^{t} z_N  (1,s)\check{w}_N(1,s)ds\right| \leq  |\xi_N (0)|+|\xi_N(t)|\leq   \gamma^2|\tilde{\zeta} (0)|+    \sqrt{\frac{2 M_{0}}{|b|}}\gamma .
\end{equation}

 {\it (2),  Weak convergence.}
     \dref{2023951539} implies that,  for any given $T>0$,
  $\{ z _N  \}$  and  $\{\check{w}_N  \}$
  are  bounded
 sequences   in  $L^{2}(0,T;H^1 (0,1))$.
 Therefore,
  there exist    subsequences   of  $\{ z _N \}$  and  $\{\check{w}_N  \}$
    which we still denote respectively by  $\{ z _N \}$ and  $\{\check{w}_N  \}$
   that satisfy
  \begin{equation} \label{2023951542}
  \left\{\begin{array}{l}
   \disp   z _N  \to   z    \ \ \mbox{weakly  in}\ \   \ L^{2}(0,  T; H ^1(0,1) ),\crr
    \disp  \check{w}_N \to  \check{w}   \ \ \mbox{weakly  in}\ \   \ L^{2}(0,  T; H ^1(0,1) ),
   % \disp     z _N(1,\cdot)\to  z (1,\cdot) \ \ \mbox{weakly  in}\ \   \ L^{2}( 0, t ),\crr
 %       \disp    \check{w}_N(1,t)\to \check{w}(1,t) \ \ \mbox{weakly  in}\ \   \ L^{2}( 0, T ),\crr
  %  \disp  z _N(0,\cdot)\to  z (0,\cdot) \ \ \mbox{weakly  in}\ \   \ L^{2}( 0, t),
   \end{array}\right.\mbox{as}\ \ N\to+\infty,
\end{equation}
where $  z (\cdot,t), \check{w}(\cdot,t) \in C(0,T;H ^1(0,1) )$.
 Using \dref{2023951539} again, the sequences
  $\{ z  _N(1,\cdot)\}$  and
  $\{\check{w} _N(1,\cdot)\}$
  are  bounded
    in   $L^{2}(0,T)$. By \dref{2023951542}, there exist subsequences of
  $\{ z  _N(1,\cdot)\}$  and
  $\{\check{w} _N(1,\cdot)\}$
   which we still denote respectively by   $\{ z  _N(1,\cdot)\}$  and
  $\{\check{w} _N(1,\cdot)\}$
   that satisfy
  \begin{equation} \label{2023951542check}
  \left\{\begin{array}{l}
   \disp     z _N(1,\cdot)\to  z (1,\cdot) \ \ \mbox{weakly  in}\ \   \ L^{2}( 0, T ),\crr
        \disp    \check{w}_N(1,\cdot)\to \check{w}(1,\cdot) \ \ \mbox{weakly  in}\ \   \ L^{2}( 0, T ),
   \end{array}\right.\mbox{as}\ \ N\to+\infty.
\end{equation}
 As a consequence of \dref{202310131615}, there exists a subsequence of $\{\xi_N\}$
 which we
still denote   by  $\{\xi_N\}$ that satisfies
\begin{equation}\label{202310231458}
 \xi_N(t)\to \xi(t) \ \ \mbox{as}\ \ N\to+\infty,\ \ \forall\ t\geq0.
\end{equation}

  {\it (3),  Strong  convergence.}
  Let $( z _N(\cdot,t),\xi _{N}(t)) ,   ( z _L(\cdot,t),\xi _{L}(t))$  be two   approximate solutions, and without loss of generality, we assume $N > L$. Set
$$\left\{  \begin{array}{l}
 \disp z _{NL}(\cdot,t) = z _N(\cdot,t)- z _L(\cdot,t)=
\gamma\sum_{j=1}^N[ \tilde{w} _{N,j}(t)- \tilde{w} _{L,j}(t)]\phi_j,\crr
\disp \xi _{NL}(t)=
\xi _{N}(t)-\xi _{L}(t),
\end{array}\right.
$$
 with the understanding that
$ \tilde{w} _{L,j}(t)\equiv 0$ when $j>L$. Then $ z _{NL}(\cdot,t)$, $\check{w}_{NL}(\cdot,t)$   and $\xi _{NL}(t)$ satisfy
 \begin{equation} \label{20231013856}
\left\{  \begin{array}{l}
\disp \langle  z _{NLt}  (\cdot,t),\phi_j\rangle_{\H}
+ \langle  z _{NLx}(\cdot,t)  , \phi_{jx}\rangle_{\H}
\disp =
    \phi_j(1)      \left[  \Gamma_{\gamma}(t)
    - c_1     z _{NL}(1,t) \right]  ,\crr
  \disp \langle \check{w}_{NLt}  (\cdot,t),\phi_j\rangle_{\H}
+ \langle \check{w}_{NLx}(\cdot,t)  , \phi_{jx}\rangle_{\H}
\disp = \phi_j(1)      \left[
     \frac{c_1}{\gamma}      z _{NL}(1,t) -c_0\check{w}_{NL}(1,t)\right]  ,\crr
     \dot{ \xi  }_{NL}(t)=    -\mbox{sgn}(b)\gamma(q+c_0)   \left[ z _N (1,t) \check{w}_N(1,t)- z _L (1,t) \check{w}_L(1,t)\right],\crr
\disp  z _{NL}(\cdot,0)=\gamma\mathcal{P}_N \tilde{w} (\cdot,0)-\gamma\mathcal{P}_L \tilde{w} (\cdot,0), \ \check{w}_{NL}(\cdot,0)= \mathcal{P}_N \check{w} (\cdot,0)- \mathcal{P}_L \check{w} (\cdot,0),\crr
 \xi  _{NL}(0)=0,\  j=1,2\cdots ,N,
 \end{array}\right.
\end{equation}
  where
\begin{equation*}\label{20231013903}
\begin{array}{l}
\Gamma_{\gamma}(t)= \disp  \frac{ b (q+c_0)}{\gamma} \left[  {\xi}_{NL}(t) \check{w}_N(1,t) +{\xi}_L(t)   \check{w}_{NL}(1,t)\right].
\end{array}
 \end{equation*}
  Let
 \begin{equation}\label{2023931622NL1012}
 L_{NL}(t)= \frac{\varepsilon}2\int_{0}^{1} z  _{NL}^2(x,t)dx +\frac{|b|}{2} \xi _{NL}^2(t)
 +  \frac{1} 2\int_{0}^{1}\check{w}_{NL}^2(x,t)dx,
 \end{equation}
  where $\varepsilon $ is a positive constant that is  small enough.
  Multiplying the  first two equations  of  \dref{20231013856} by
$ \varepsilon\gamma[ \tilde{w}  _{N,j}(t)- \tilde{w}  _{L,j}(t)]$ and
 $\check{w}_{N,j}(t)-\check{w}_{L,j}(t)$, respectively,
  and summing  for $j = 1, \cdots ,N$,
  and
  multiplying the third    equation in    \dref{20231013856}
 by $ \xi _{NL}(t)$, we get
   \begin{equation}\label{20231092217}
   \begin{array}{l}
 \disp  \dot{L}_{NL}(t) \leq\disp
    - \varepsilon \| z  _{NLx}(\cdot,t)\|_{\H}^2- \|\check{w}_{NLx}(\cdot,t)\|_{\H}^2
  \disp -c_0 \check{w}_{NL}^2(1,t)-c_1 \varepsilon  z  _{NL}^2(1,t)\crr
  \disp+\left[\frac{c_1}{\gamma}+   \varepsilon (q+c_0)\sqrt{ {2|b| M_{0}} }\right]  |\check{w}_{NL}(1,t) z  _{NL}(1,t)|
  + [I_{1}(t)+I_{2}(t)]|\xi  _{NL}(t)|
  \end{array}
 \end{equation}
 where
 \begin{equation}\label{202310231520}
  \left\{ \begin{array}{l}
  \disp  I_{1}(t)=  |b|\gamma(q+c_0)|z_N(1,t)\check{w}_N(1,t)-z_L(1,t)\check{w}_L(1,t)| ,\crr
  \disp  I_{2}(t)=\frac{\varepsilon |b|(q+c_0)}{\gamma}|\check{w}_N(1,t)z_{NL}(1,t)| .
  \end{array}\right.
 \end{equation}
    For any $T>0$ and any $ 0<t\leq T$, $\xi_{NL}$ is continuous on $[0,t]$. Hence, there exists a
 $t_*\in [0,t]$ such that
 \begin{equation}\label{202310231538}
     |\xi_{NL}(t_*)|=\max_{s\in [0,t]}|\xi_{NL}(s)|.
 \end{equation}
Hence, it follows from \dref{202310231510} that there exists a positive constant $M_1$ that is independent of $N,L$ and $t$ such that
 \begin{equation}\label{202310231529}
  \int_{0}^{t} I_{1}(s)|\xi_{NL}(s)|ds\leq
    |\xi_{NL}(t_*)|\int_{0}^{t}I_1(s)ds\leq
    M_1 |\xi_{NL}(t_*)|.
 \end{equation}
 Combining \dref{202310231538},  \dref{2023951539} and \dref{202310231520}, there exists a positive constant $M_2$ that is independent of $N,L$ and $t$ such that
 \begin{equation}\label{202310231550}
 \begin{array}{l}
 \disp \int_{0}^{t} I_{2}(s)|\xi_{NL}(s)|ds\leq
    |\xi_{NL}(t_*)|\int_{0}^{t}I_2(s)ds\crr
    \disp \leq
    \frac{\varepsilon|b|(q+c_0)}{2\gamma}|\xi_{NL}(t_*)|\int_{0}^{t}\check{w}_N^2(1,s)+z_{NL}^2(1,s) ds
     \leq
    M_2|\xi_{NL}(t_*)|.
    \end{array}
 \end{equation}
    By Young's inequality,  for any $  \delta>0 $, we have
  \begin{equation}\label{202310142306}
    |\check{w}_{NL}(1,t) z  _{NL}(1,t)|\leq \frac{1}{2\delta}\check{w}_{NL}^2(1,t)+
   \frac{\delta}{2} z^2_{NL}(1,t) .
 \end{equation}
We choose
  \begin{equation}\label{202310142308}
  0<\delta<\frac{\sqrt{2}c_1}{(q+c_0)\sqrt{  |b| M_{0}  }},\ \
  0<\varepsilon<\frac{\sqrt{2}c_0\delta}{ (q+c_0)\sqrt{|b| M_{0}}    }
 \end{equation}
 and choose $\gamma$ large enough such that
  \begin{equation}\label{202310142315}
 0<\frac{c_1}{2\gamma }<\min \left\{  c_0 -\frac{\varepsilon(q+c_0)\sqrt{   |b| M_{0} }}{\sqrt{2}\delta},  c_1\varepsilon-\frac{\delta\varepsilon(q+c_0)\sqrt{  |b| M_{0}}  }{\sqrt{2}} \right\}.
 \end{equation}
Substituting the inequalities
\dref{202310231529}, \dref{202310231550},
 \dref{202310142306},  \dref{202310142308} and \dref{202310142315} into \dref{20231092217}, we get
  %\begin{equation}\label{202310142313}
%   \begin{array}{ll}
% \disp  \dot{L}_{NL}(t) \leq&\disp
%    - \varepsilon \| z  _{NLx}(\cdot,t)\|_{\H}^2- \|\check{w}_{NLx}(\cdot,t)\|_{\H}^2
%  \disp -c_0 \check{w}_{NL}^2(1,t)-c_1 \varepsilon  z  _{NL}^2(1,t)\crr
%  &\disp+\left[\frac{c_1}{\gamma}+   \varepsilon (q+c_0)\sqrt{ {2|b| M_{0}} }\right]  \check{w}_{NL}(1,t) z  _{NL}(1,t)+ {C}_T \xi  _{NL}(t) ,
%  \end{array}
% \end{equation}
% Choosing $\gamma$ large enough,  \dref{20231092217} yields
%%  \begin{equation}\label{20231092228}
%% \dot{L}_{NL}(t)\leq -\varepsilon\| z  _{NLx}(\cdot,t)\|_{\H}^2- \|\check{w}_{NLx}(\cdot,t)\|_{\H}^2
%%  + (M_1+M_2) |\xi  _{NL}(t_*)|  .
%% \end{equation}
%%Consequently,
 \begin{equation}\label{20231013939}
  {L}_{NL}(t)+\int_{0}^{t}\varepsilon\| z  _{NLx}(\cdot,s)\|_{\H}^2+ \|\check{w}_{NLx}(\cdot,s)\|_{\H}^2ds\leq  {L}_{NL}(0)+
     (M_1+M_2)  |\xi  _{NL}(t_*)|
 \end{equation}
 for any $t\in [0,T]$.
 Since
   $\xi_{NL}(t_*)\to 0$ as $N,L\to+\infty$,
it follows from   \dref{20189101501Ad92},  \dref{202310131547} and \dref{20231013939} that
  \begin{equation}\label{20231013944}
  {L}_{NL}(t) +\int_{0}^{t}\varepsilon\| z  _{NLx}(\cdot,s)\|_{\H}^2+ \|\check{w}_{NLx}(\cdot,s)\|_{\H}^2ds\to 0\ \ \mbox{as}\ \ N,L\to+\infty,\ \ \forall\ t\in[0,T],
 \end{equation}
 which, together with \dref{202310131547} and \dref{2023931622NL1012}, implies that  $\{\tilde{w}_N(\cdot,t) \}$ and $\{\check{w}_N(\cdot,t) \}$
  are   two  Cauchy sequences  in    $ \H$;  $\{\tilde{w}_N  \}$ and $\{\check{w}_N  \}$
  are   two  Cauchy sequences  in    $ L^2(0,t;H^1(0,1))$. Notice that $z_{NL}(0,t)=\check{w}_{NL}(0,t)=0$, it follows from the Sobolev trace-embedding that
   \begin{equation}\label{20231014901}
    z_{NL}^2(1,t)\leq  \int _0^1z  _{NLx}^2(x,t)dx,\ \ \ \check{w}_{NL}^2(1,t)\leq   \int_0^1\check{w}_{NLx}^2(x,t)dx,
   \end{equation}
which, together with \dref{202310131547} and \dref{20231013944}, implies that  $\{\tilde{w}_N(1,t) \}$ and $\{\check{w}_N(1,t) \}$
  are   two  Cauchy sequences  in   $L^2(0,t)$ for any $t\in [0,T]$.
  In particular,  we have
   \begin{equation}\label{20239201457}
 \int_0^t
  \tilde{w}_N (1,s) \check{w}_N(1,s)ds
 \to \int_0^t
  \tilde{w}  (1,s) \check{w} (1,s)ds  \ \ \mbox{as}\ \ N\to+\infty.
\end{equation}
Moreover,
 \dref{202310131615} and  \dref{202310231458}  yield
 \begin{equation}\label{20239201610}
  \lim_{N\to+\infty }\int_0^t\xi _N(s)   \check{w}_N(1,s)ds=
   \int_0^t\xi  (s)   \check{w} (1,s) ds.
  \end{equation}

 {\it (4),  Passage to the limit.}
 For any  $0<t<T$,  by integrating  \dref{2023818918Ad92} from $0$ to $t$,
we obtain
 \begin{equation} \label{20239201614}
\left\{ \begin{array}{l}
\disp \langle \tilde{w}_N   (\cdot,t),\phi_j\rangle_{\H}
+ \int_0^t\langle \tilde{w}_{Nx}(\cdot,\tau)  , \phi_{jx}\rangle_{\H}d\tau
\disp =\langle \tilde{w}_N   (\cdot,0),\phi_j\rangle_{\H}\crr
\disp  + \phi_j(1) \int_0^t     \left[   b (q+c_0) \tilde{\zeta}_N (\tau)\check{w}_N(1,\tau)
    - c_1     \tilde{w}_N (1,\tau)     \right]d\tau  ,\crr
  \tilde{\zeta}_N  (t)=\disp \tilde{\zeta}  (0)  -\mbox{sgn}(b)(q+c_0) \int_0^t \tilde{w} _N (1,\tau)\check{w}_N(1,\tau)d\tau,\crr
    \disp \langle \check{w}_N   (\cdot,t),\phi_j\rangle_{\H}
+ \int_0^t\langle \check{w}_{Nx}(\cdot,\tau)  , \phi_{jx}\rangle_{\H}d\tau
\disp =\langle \check{w}_N   (\cdot,0),\phi_j\rangle_{\H}\crr
\disp   +\phi_j(1)   \int_0^t      \left[
     c_1     \tilde{w}_N (1,\tau) -c_0\check{w}_N (1,\tau)\right]  d\tau.
    \end{array}\right.
\end{equation}
In view of \dref{2023951542}, \dref{2023951542check},   \dref{202310231458}, \dref{20239201457},
\dref{20239201610}, \dref{20189101501Ad92}  and \dref{202310131547}, we  pass   to the limit  as  $N\to+\infty$ in
\dref{20239201614} to get
  \begin{equation} \label{20239161659}
\left\{  \begin{array}{l}
\disp \langle \tilde{w}   (\cdot,t),\phi_j\rangle_{\H}
+ \int_0^t\langle \tilde{w}_{x}(\cdot,\tau)  , \phi_{jx}\rangle_{\H}d\tau
\disp =\langle \tilde{w}   (\cdot,0),\phi_j\rangle_{\H}\crr
\disp  + \phi_j(1) \int_0^t     \left[   b (q+c_0) \tilde{\zeta} (\tau)\check{w}(1,\tau)
    -c_1     \tilde{w} (1,\tau)     \right]d\tau  ,\crr
  \tilde{\zeta}  (t)=\disp \tilde{\zeta} (0)  -\mbox{sgn}(b)(q+c_0) \int_0^t \tilde{w} (1,\tau)\check{w}(1,\tau)d\tau,\crr
  \disp \langle \check{w}   (\cdot,t),\phi_j\rangle_{\H}
+ \int_0^t\langle \check{w}_{x}(\cdot,\tau)  , \phi_{jx}\rangle_{\H}d\tau
\disp =\langle \check{w}   (\cdot,0),\phi_j\rangle_{\H}\crr
\disp      +\phi_j(1)   \int_0^t      \left[
     c_1     \tilde{w} (1,\tau) -c_0\check{w} (1,\tau)\right]  d\tau,
 \end{array}\right. \
 j=1,2\cdots  ,
\end{equation}
which implies   that $ (\tilde{w},\tilde{\zeta},\check{w} )  $ is a weak solution of the system \dref{20238292127}.
 Since the weak convergence \dref{2023951542} and   \dref{2023951542check},
     \dref{2023951709}  follows from  \dref{2023951539}, \dref{202310131547}  and    the weakly lower semicontinuity of the norm $\|\cdot\|_{L^2(0,\infty)}$  \cite[p.6, Theorem 1.1.1]{EvansweakConv}.
    \end{proof}

\begin{lemma}\label{Lm2023951644}
Suppose that  $b\neq 0$,  $c_0 >0$ and $c_1>0$. Then for any initial state
$(\tilde{w}(\cdot,0),\check{w}(\cdot,0),\tilde{\zeta}(0))\in \H^2\times\R$,
system
  \dref{20238292127} admits a   weak solution $(\tilde{w}(\cdot,t),\check{w}(\cdot,t),\tilde{\zeta}(t))\in C(0,\infty;\H^2\times\R)$ such that
\begin{equation}\label{2023951645}
   |\tilde{\zeta}(t)-  {\zeta}_* |+\|(\tilde{w}(\cdot,t),\check{w}(\cdot,t))\|_{\H^2} \to 0 \ \ \mbox{as}\ \ t\to+\infty,
\end{equation}
%and
%\begin{equation}\label{2023951646}
% \tilde{\zeta}(t) \to  {\zeta}_* \ \ \mbox{as}\ \ t\to+\infty,
% \end{equation}
 where $ {\zeta}_*$  is a constant that may not be $0$.

\end{lemma}
\begin{proof}
 By  Lemma \ref{Lm2023828938Ad92}, system
  \dref{20238292127} admits a  weak solution $(\tilde{w}(\cdot,t),\check{w}(\cdot,t),\tilde{\zeta}(t))\in C(0,\infty;\H^2\times\R)$ such that
 \dref{2023951709} holds.
 % Now, we prove   the stability.
If we let $u_0(t)=-(q + c_0 )\check{w}(1,t)$, then
$u_0\in L^2(0,\infty)$.    Since the $\tilde{w},\tilde{\zeta}$-subsystem of \dref{20238292127}    happens to be system \dref{20237122040frac1b},     by exploiting Lemma \ref{Lm2023811150}, there exists a constant $\zeta_*$ that may not be zero  such that
 \begin{equation}\label{2023951705}
 \|\tilde{w}(\cdot,t)\|_{\H} +|\tilde{\zeta}(t)-\zeta_*|\to 0 \ \ \mbox{as}\ \ t\to+\infty.
 \end{equation}
Moreover, $\tilde{w}(0,\cdot),\tilde{w}(1,\cdot)\in L^2(0,\infty) $.
Since the
$\check{w}$-subsystem of \dref{20238292127}
is a well-known exponentially stable heat equation with inhomogeneous terms $q^2e^{qx}\tilde{w}(0,t)$,$-q\tilde{w}(0,t)$ and $c_1\tilde{w}(1,t)$,
 we can  deduce  from   \cite[Lemma 3.1]{FengAnnual} that   $\|\check{w}(\cdot,t)\|_{\H} \to 0$   as $t\to+\infty$.
%Therefore, if we define
%\begin{equation}\label{2023931622Ad95}
% P (t)=\frac{1} 2\int_{0}^{1}\check{w} ^2(x,t)dx,
% \end{equation}
% then the $\check{w}$-subsystem of \dref{20238292127} satisfies
%    \begin{equation}\label{2023818940Ad95**}
%     \begin{array}{ll}
%\dot{P} (t)&\disp = \disp -\int_{0}^{1} \check{w}_{x}^2(x,t)dx
%\disp+
%q \check{w} (0,t)\tilde{w}(0,t)-c_0 \check{w}^2(1,t)+c_1 \check{w}(1,t)\tilde{w}(1,t)\crr
%&\leq  \disp -\omega_1 P(t)
%\disp+C\left[ \tilde{w} ^2(0,t) +
%      \check{w} ^2(1,t)+  \tilde{w} ^2(1,t)  \right] ,
%%q \check{w} (0,t)\tilde{w}(0,t)-c_0 \check{w}^2(1,t)+c_1 \check{w}(1,t)\tilde{w}(1,t),
%\end{array}
%\end{equation}
%where  $\omega_1$ and $C$ are    positive constants.
%% \begin{equation}\label{20238291501Ad95}
%% \begin{array}{l}
%% \disp f(t)=  \frac{q}2  \left[\check{w} ^2(0,t)+ \tilde{w} ^2(0,t)\right]+
%%     \frac{c_1}2\left[\check{w} ^2(1,t)+  \tilde{w} ^2(1,t)  \right].
%%     \end{array}
%%\end{equation}
%Thanks to  \dref{2023951709} and the fact $\tilde{w}(0,\cdot) \in L^2(0,\infty) $, $f\in L^1(0,\infty)$.
%  Applying  Lemma \ref{Lm20231004} in Appendix to the inequality \dref{2023818940Ad95**}, we get
%  $\lim_{t\to+\infty} P(t)=0$,
%% \begin{eqnarray}\label{APPP5Ad921}
%% \lim_{t\to\infty} P(t)=0,
%%\end{eqnarray}
% which, together with \dref{2023931622Ad95},   implies $\|\check{w}(\cdot,t)\|_{\H}\to 0$ as $t\to+\infty$.
  \end{proof}

{\it Proof of Theorem \ref{th20237141510}.}
By Lemma \ref {Lm2023951644},
the transformed system
  \dref{20238292127} admits a  weak solution
  $(\tilde{w}(\cdot,t),\check{w}(\cdot,t),\tilde{\zeta}(t))\in C(0,\infty;\H^2\times\R)$
 such that  \dref{2023951645} holds.
 In order to relate the transformed system
  \dref{20238292127} to the  original system   \dref{2023720848}, we define the transformation
 $ \Pi : \H\to \H $ by
  \begin{equation}\label{2023813837}
( \Pi f)(x)=f(x)+q\int_0^xe^{q(x-s)}f(s)ds,\     \ \forall\ f\in \H .
 \end{equation}
 Then, $ \Pi \in \mathcal{L}(\H )$ is invertible and more specifically
   \begin{equation}\label{2023813841}
( \Pi ^{-1}g)(x)=g(x)-q\int_0^xg(s)ds,\    \ \forall\ g\in \H .
 \end{equation}
In terms of the constant $b$, we define the transformation
 $ \Upsilon_b : \R\to \R$ by
  \begin{equation}\label{2023951821}
 \Upsilon_b (s)  = \frac1b-s,\ \    \forall\ s\in \R.
 \end{equation}
 Clearly,
 %  $ \Upsilon_b $ is invertible and more specifically
$  \Upsilon_b ^{-1}  =\Upsilon_b$.

 Let
 \begin{equation}\label{2023951826}
\begin{pmatrix}
  w(\cdot,t) \\
  \hat{w}(\cdot,t) \\
  \zeta(t)
\end{pmatrix}
=\begin{pmatrix}
                1&-1&0 \\
                0&\Pi&0 \\
                0&0&\Upsilon_b
              \end{pmatrix}^{-1}\begin{pmatrix}
  \tilde{w}(\cdot,t) \\
  \check{w}(\cdot,t) \\
  \tilde{\zeta}(t)
\end{pmatrix}
=\begin{pmatrix}
                1&\Pi^{-1}&0 \\
                0&\Pi^{-1}&0 \\
                0&0&\Upsilon_b
              \end{pmatrix} \begin{pmatrix}
  \tilde{w}(\cdot,t) \\
  \check{w}(\cdot,t) \\
  \tilde{\zeta}(t)
\end{pmatrix}.
 \end{equation}
A straightforward computation shows that such a defined
  $( {w}(\cdot,t),\hat{w}(\cdot,t), {\zeta}(t))\in C(0,\infty;\H^2\times\R)$
is a weak solution of system  \dref{2023720848}. Moreover,
 we can obtain the convergence    \dref{2023720844w}
   directly by combining   \dref{2023951826} and \dref{2023951645}.
  \hfill$\Box$

  \section{Proof of Theorem \ref{th2023829953}}\label{PfTh3}
We first consider the following transformed system:
    \begin{equation}\label{2023952151}
\left\{\begin{array}{l}
 \tilde{z}_{t}(x,t)=\tilde{z} _{xx}(x,t)   ,
 \crr
 \tilde{z} _{x} (0,t)= 0 ,\
  \disp\tilde{z} _{x}(1,t)=- b \tilde{\zeta}(t)  u_0(t)-c_1\tilde{z} (1,t),    \crr
\disp\dot{\tilde{\zeta}}(t)=   \mbox{sgn}(b)  \tilde{z} (1,t)\left[d(t)-  (q+c_0)\check{z}(1,t)  \right] ,\ \ d\in L^{\infty}(0,\infty),\crr
  \check{z}_{t}(x,t)=\check{z}_{xx}(x,t)   ,\crr
 \check{z}_{x} (0,t)=-q\tilde{z}(0,t),\
\disp  \check{z}_{x}(1,t)= - c_0  \check{z}(1,t) +c_1  \tilde{z}(1,t)  .
%u_0(t)= d(t)-  (q+c_0)\check{z}(1,t) ,\ \ d\in L^{\infty}(0,\infty).
 \end{array}\right.
\end{equation}
%where $d\in L^{2}_{\rm loc}(0,\infty)$.

\begin{lemma}\label{Lm2023952250}
Let $b\neq 0$,  $c_0 >0$,  $c_1>0$ and
$d\in L^{\infty}(0,\infty)$.
  Then for any initial state
$(\tilde{z}(\cdot,0),\check{z}(\cdot,0),\tilde{\zeta}(0))\in \H^2\times\R$,
system
  \dref{2023952151} admits a   weak solution $(\tilde{z}(\cdot,t),\check{z}(\cdot,t),\tilde{\zeta}(t))\in C(0,\infty;\H^2\times\R)$ such that
  \begin{equation}\label{2023961016}
    \int_0^{\infty}\left[\check{z}^2(0,t) + \tilde{z}^2(0,t) + \tilde{z}^2(1,t)+\check{z}^2(1,t)  \right]
    dt<+\infty,
    \end{equation}
  and
\begin{equation}\label{2023952251}
 \sup_{t\geq0} \left[ \|(\tilde{z}(\cdot,t),\check{z}(\cdot,t))\|_{\H^2}+|\tilde{\zeta}(t) |\right]< +\infty.
\end{equation}
If we assume further that
  there exists a time $\tau>0$ such that
%$d\in L^{\infty}(0,\infty)$ satisfies
\begin{equation}\label{202381809*d}
      \lim_{t\to+\infty}\int_{t}^{t+\tau}
       d (s) ds \neq0,
\end{equation}
then
 \begin{equation}\label{2023952252}
  \tilde{\zeta}(t) \to 0 \ \ \mbox{as}\ \ t\to+\infty.
  \end{equation}

\end{lemma}
 \begin{proof}
   By the  Galerkin approximation,
     for any initial state
$(\tilde{z}(\cdot,0),\check{z}(\cdot,0),\tilde{\zeta}(0))\in \H^2\times\R$,
system
\dref{2023952151}  admits a   weak solution $(\tilde{z}(\cdot,t),\check{z}(\cdot,t),\tilde{\zeta}(t))\in C(0,\infty;\H^2\times\R)$ such that
 \begin{equation}\label{2023951709zz}
   \int_{0}^{\infty}  \int_{0}^{1}\tilde{z}_{x}^2(x,t)+\check{z}_{x}^2(x,t)dx+\tilde{z}^2(1,t)+\check{z}^2(1,t)dt<+\infty,
 \end{equation}
which leads to \dref{2023961016} directly due to the Sobolev trace-embedding.
 Since the proof    is almost the same as Lemma  \ref{Lm2023828938Ad92}, we omit   it due to the space limitation. We only consider the stability.

Let  $u_0(t)=d(t)-  (q+c_0)\check{z}(1,t)$.
Since  \dref{2023951709zz} and \dref{2023952151},    we have
  $d\in L^{\infty}(0,\infty)$ and $\check{z}(1,\cdot)\in L^2(0,\infty)$.
  By exploiting
Lemma \ref{Lm2023811150} that $\tilde{z}(0,\cdot)\in L^2(0,\infty)$ and
  \begin{equation}\label{2023811157zz}
 \|\tilde{z}(\cdot,t)\|_{\H} +|\tilde{\zeta} (t)- \zeta_*|\to 0 \ \ \mbox{as}\ \ t\to+\infty,
 \end{equation}
 where $\zeta_*$ is a constant that may not be zero.
 Since the $\check{z}$-subsystem of \dref{2023952151} is a well-known exponentially stable heat equation with inhomogeneous terms $-q\tilde{z}(0,t)$ and $c_1\tilde{z}(1,t)$,
 we can conclude from   \cite[Lemma 3.1]{FengAnnual} that   $\|\check{z}(\cdot,t)\|_{\H} \to 0$   as $t\to+\infty$.
%  \begin{equation}\label{20239211535}
% \|\check{z}(\cdot,t)\|_{\H}^2 \to 0 \ \ \mbox{as}\ \ t\to+\infty,
% \end{equation}
As a result, the boundedness  \dref{2023952251}  follows immediately from
  \dref{2023811157zz}.

By \dref{2023951709zz},   it follows that
\begin{equation}\label{20239211456}
      \lim_{t\to+\infty}\int_{t}^{t+\tau}
       \check{z}^2(1,s) ds=0 ,\ \ \forall\ \tau>0.
\end{equation}
H\"{o}lder's inequality yields
\begin{equation}\label{20239211511}
  \int_{t}^{t+\tau}
       |\check{z} (1,s)| ds  \leq  \sqrt{\tau} \left(\int_{t}^{t+\tau}
       \check{z}^2(1,s) ds\right)^{1/2}  ,\ \ \forall\ \tau>0.
\end{equation}\
Consequently, it follows from \dref{20239211456},  \dref{20239211511}  and the assumption \dref{202381809*d} that
\begin{equation}\label{202396802}
      \lim_{t\to+\infty}\int_{t}^{t+\tau}
       u_0 (s) ds=\lim_{t\to+\infty}\int_{t}^{t+\tau}
      d (s) ds \neq0.
\end{equation}
 Note that the $\tilde{z},\tilde{\zeta}$-subsystem of \dref{2023952151} is equivalent to   system \dref{20237122040frac1b},  we employ   Lemma \ref{Lm2023811150}   to get  $\zeta_*=0$. So \dref{2023952252} holds.
   \end{proof}

{\it Proof of Theorem \ref{th2023829953}.}
Let $d(t)=v_x(1,t)$, where
$v_x(1,t)$  is given by
  \dref{2023828927}. Owing to the assumption \dref{2023829957},
  $\|d\|_{L^{\infty}(0,\infty)}<+\infty$ and
   \begin{equation}\label{2023961052}
   \sup_{t\geq 0}\|v(\cdot,t)\|_{\H}<+\infty,
   \end{equation}
  where $v$ is given by \dref{2023829858}.
By Lemma \ref {Lm2023952250},
the transformed system
  \dref{2023952151} admits a   weak solution
  $(\tilde{z}(\cdot,t),\check{z}(\cdot,t),\tilde{\zeta}(t))\in C(0,\infty;\H^2\times\R)$
 such that  \dref{2023961016} and \dref{2023952251}   hold.
In terms of the operators $\Pi$ and  $\Upsilon_b$ that are given   by \dref{2023813837} and \dref{2023951821} respectively, we define
 \begin{equation}\label{2023952235}
\begin{pmatrix}
  z(\cdot,t) \\
  \hat{z}(\cdot,t) \\
  \zeta(t)
\end{pmatrix}
 =\begin{pmatrix}
                1&\Pi^{-1}&0 \\
                0&\Pi^{-1}&0 \\
                0&0&\Upsilon_b
              \end{pmatrix} \begin{pmatrix}
  \tilde{z}(\cdot,t) \\
  \check{z}(\cdot,t) \\
  \tilde{\zeta}(t)
\end{pmatrix}.
 \end{equation}
A direct computation shows that  such a defined
 $( {z}(\cdot,t),\hat{z}(\cdot,t), {\zeta}(t))\in C( 0,\infty ;\H^2\times\R)$
is governed by
\begin{equation}\label{2023952148}
  \left\{\begin{array}{l}
   \disp  z_t(x,t)=z_{xx}(x,t),   \crr
  \disp  z_x(0,t)=
    -qz(0,t) , \ z_x(1,t)=b\zeta(t)u_0(t)-v_x(1,t),\crr
  \hat{z}_{t}(x,t)=\hat{z}_{xx}(x,t)   , \; x\in (0,1),
 \crr
 \hat{z}_{x} (0,t)=-q {z}(0,t),\  \
\hat{z}_{x}(1,t)= u_0(t)-v_x(1,t)+c_1 [z(1,t)-\hat{z}(1,t)] ,    \crr
 \dot{\zeta}(t)= -\mbox{sgn}(b)[z(1,t)-\hat{z}(1,t)]u_0(t) ,\crr
 u_0(t) \disp =   \disp  -(q+c_0)\left[  \hat{z}(1,t)+q \int_{0}^{1}e^{q(1-x)}\hat{z}(x,t)dx\right]+ v_x(1,t).
        \end{array}\right.
\end{equation}

Let
  \begin{equation}\label{202396916}
w(x,t)=z(x,t)+v(x,t), \ \ x\in[0,1],\ t\geq0,
\end{equation}
where $v$  is given by \dref{2023829858}.
By using \dref{2023829858}, \dref{20221131932}  and \dref{2023952148},
  a straightforward computation shows that
$( w(\cdot,t),\hat{z}(\cdot,t), {\zeta}(t))\in C(0,\infty;\H^2\times\R)$
is a weak solution of the closed-loop system \dref{2023829945}.
Note that
\begin{equation}\label{2023961037}
w(0,t)-r(t)=  w(0,t)-v(0,t)=z(0,t)=\hat{z}(0,t)+\tilde{z}(0,t)=
   \check{ z}(0,t)+\tilde{z}(0,t),
\end{equation}
   \dref{2023829954}  holds due to \dref{2023961016}.
Furthermore,  the boundedness  \dref{2023829956} follows from \dref{202396916}, \dref{2023952235}, \dref{2023961052} and \dref{2023952251}.

 %By  \dref{2023829945}, \dref{2023813837} and \dref{2023952235}, we obtain
%   \begin{equation}\label{2023961108}
%\begin{array}{rl}
%   u_0(t) &\disp =   \disp  -(q+c_0)\left[  \hat{z}(1,t)+q \int_{0}^{1}e^{q(1-x)}\hat{z}(x,t)dx\right]+ v_x(1,t)\crr
%   &\disp =   \disp  -(q+c_0) \check{z}(1,t) + v_x(1,t).
% \end{array}
% \end{equation}
 When the reference $r$ satisfies \dref{202396903},  $d(t)=v_x(1,t)$  satisfies  \dref{202381809*d}. Hence, \dref{2023952252} holds due to Lemma \ref{Lm2023952250}.
By recalling \dref{2023951821} and \dref{2023951826},  it follows that $\zeta(t)=\Upsilon_b\tilde{\zeta}(t)= \frac1b- \tilde{\zeta}(t)$
 %  \begin{equation}\label{2023828927d}
% \zeta(t)=\Upsilon_b\tilde{\zeta}(t)= \frac1b- \tilde{\zeta}(t),
%\end{equation}
  which, together with  \dref{2023952252}, leads to the convergence \dref{202396902} easily.

\section{Numerical simulations}\label{NumSim}

In this section, we  present   numerical simulations  for  systems \dref{2023720848} and  \dref{2023829945constant} to validate  the  theoretical    results. The finite difference scheme is adopted in discretization.
The numerical results are programmed in Matlab. The time step and the space step are
taken as $0.0001$ and $0.02$, respectively.
The reference signal  and the  corresponding parameters  are chosen  as
 \begin{equation}\label{20181211918}
 \left.\begin{array}{l}
\disp r^* =3,\  c_0=c_1=5,\  q=2,\  b=-10.
\end{array}\right.
\end{equation}
In both cases, the initial state is chosen as
\begin{equation}\label{2023962257}
w(x,0)=qx-1, \ \hat{w}(x,0)=\hat{z}(x,0)=0, \ \zeta(0)=0.
\end{equation}

The solution  of system   \dref{2023720848} is
 plotted in
Figure \ref{Fig1}.
From Figure \ref{Fig1} we   see that
the system state has been  stabilized effectively, which shows that our controller works well.
Both the
  state  observer and the  control coefficient update law
  are convergent. However, $\zeta(t)$ has not converged to $\frac1b$. Hence, the numerical simulation is  consistent  with our theoretical results in Theorem \ref{th20237141510} and Remark \ref{Re2023831}.
  The tracking result of system
  \dref{2023829945constant} is  plotted in
  Figure \ref{Fig2} which shows that the performance output converges to the constant reference
smoothly. Since  the   persistent excitation
condition is satisfied, the estimate $\zeta$ converges to $\frac1b$ as $t\to+\infty$.  All the state of system
  \dref{2023829945constant} are bounded.

\begin{figure}[!htb]\centering
\subfigure[$w(x,t)$.]
  {\includegraphics[width=0.32\textwidth]{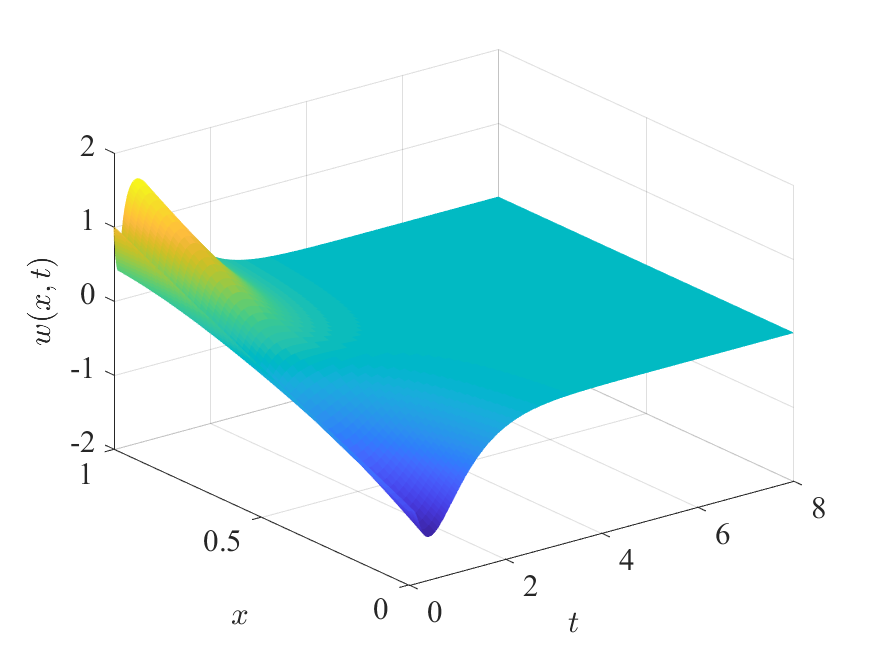}}
\subfigure[ $\hat{w}(x,t)$.]
  {\includegraphics[width=0.32\textwidth]{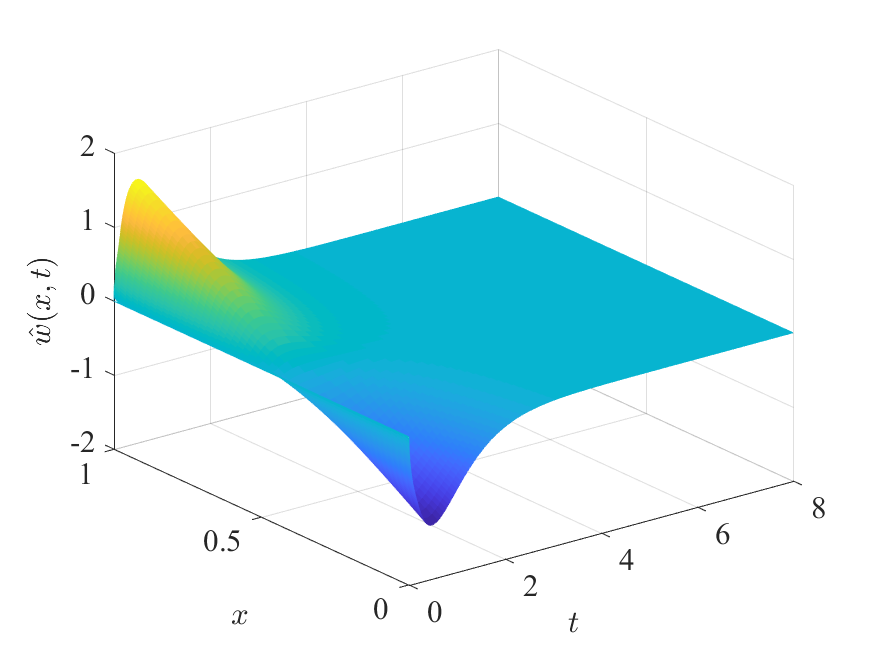}}
\subfigure[Control coefficient estimation.]
 {\includegraphics[width=0.32\textwidth]{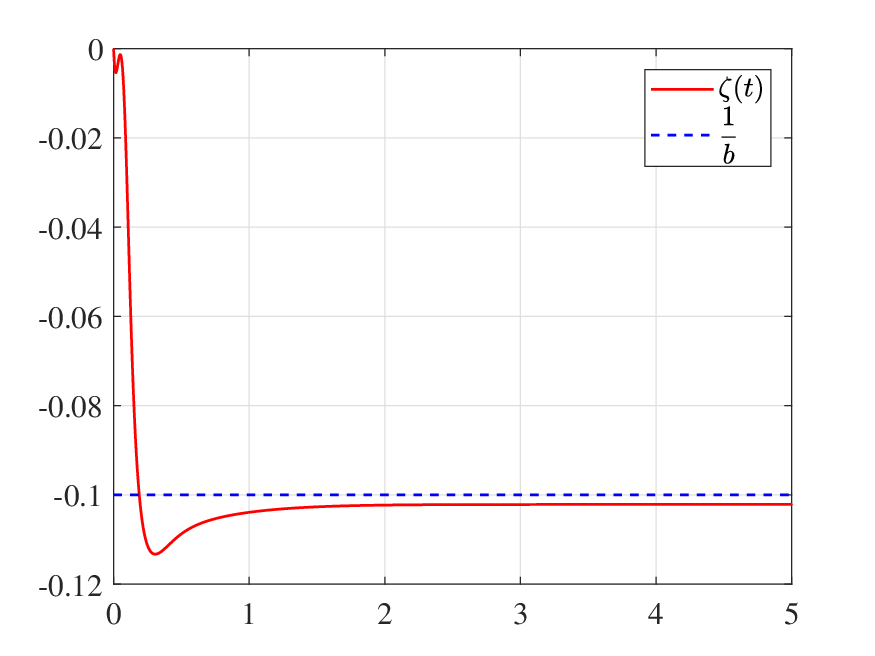}}
 \caption{Simulations for system \dref{2023720848}.}\label{Fig1}
\end{figure}

\begin{figure}[!htb]\centering
\subfigure[$w(x,t)$.]
% {\includegraphics[width=1.5in,height=3cm]{w}}
 %{\includegraphics[width=0.45\textwidth,height=0.3\textheight]{w}}
  {\includegraphics[width=0.48\textwidth]{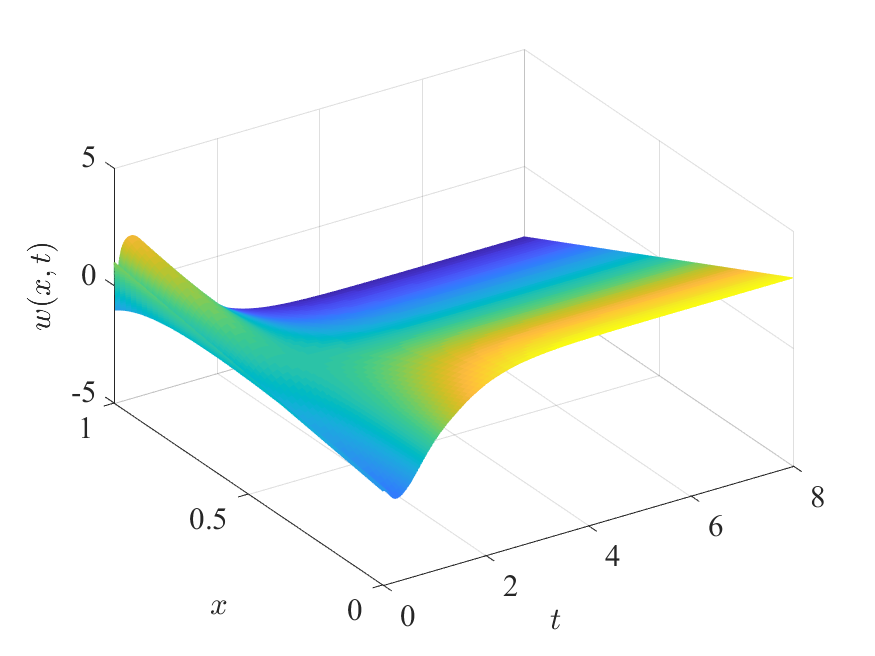}}
\subfigure[ $\hat{z}(x,t)$.]
 %{\includegraphics[width=1.5in,height=3cm]{phi}}
  {\includegraphics[width=0.48\textwidth]{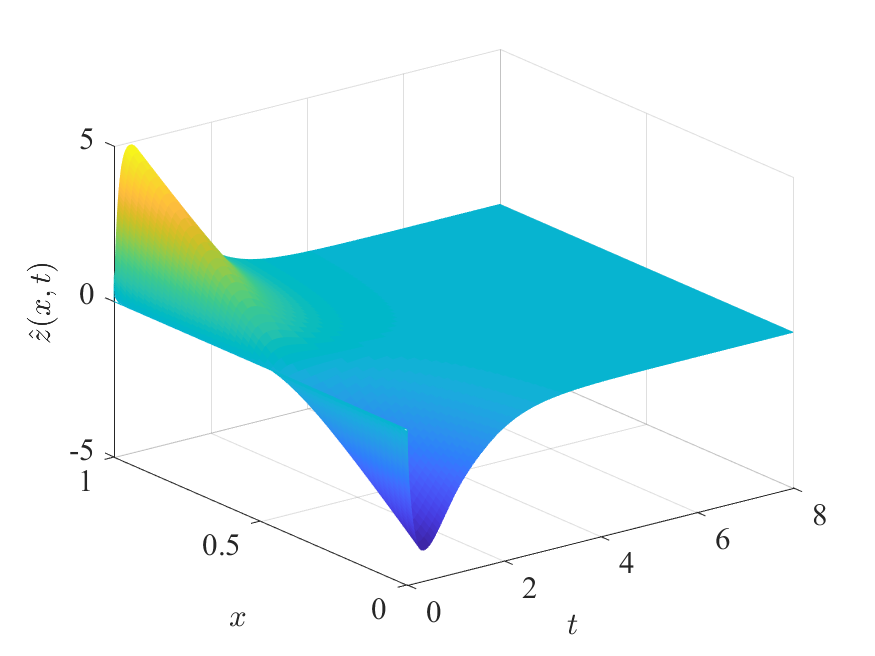}}\\
\subfigure[Control coefficient estimation.]
 {\includegraphics[width=0.48\textwidth]{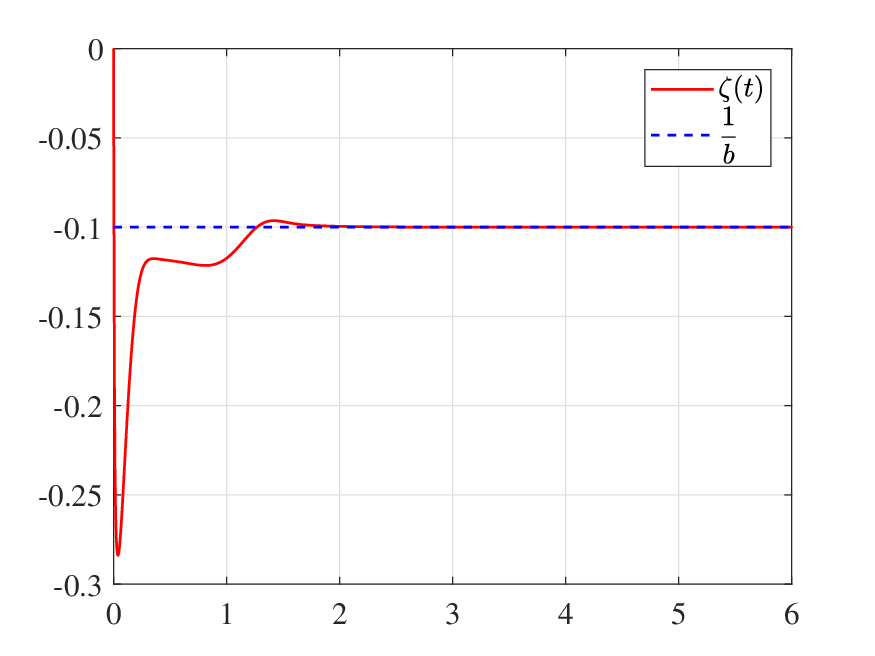}}
\subfigure[Output tracking.]
 {\includegraphics[width=0.48\textwidth]{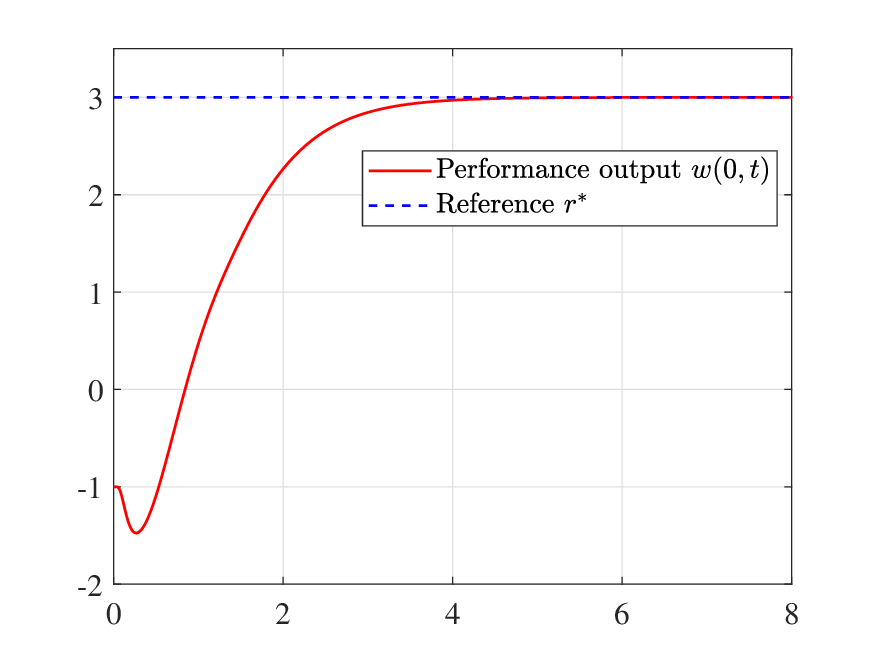}}
\caption{Simulations for system \dref{2023829945constant}.}\label{Fig2}
\end{figure}
 \section{Conclusions}\label{Concluding}
This paper
  presents a new adaptive control scheme for the
  unstable heat equation \dref{20237122035}.
 Since  the control
   coefficient is unknown, the conventional partial differential equation backstepping
  can not be used directly.  We overcome this
 obstacle by dividing the controller into two factors. One of the factors is used to compensate for
  the
 unknown control coefficient and the other  factor is used to stabilize the system.
 A state  observer,  where  both the unknown
  control coefficient and its corresponding estimate are absent,   is
  designed  to estimate the system state,  while a  new update law is proposed to estimate
   the reciprocal of   control coefficient.
 Very importantly,  the convergence of the state estimation  does not rely on the convergence of
the control coefficient estimation. This implies  that the  conventional controller design based on the
 separation principle
 is not trivial  since  we always do not know
the effective estimation of the unknown
control coefficient before the controller design.

%As one novelty  of this paper, we design the controller for the state observer
% rather than the control plant itself.

In summary, the problems that are caused by the unknown control
coefficient are addressed successfully  via stabilizing the unstable heat equation.
The successful  controller design  is mainly  attributed to  two   innovative ideas: (1) we  design the update law to estimate the    reciprocal of $b$
rather than to estimate the control coefficient $b$ itself; (2)
 we design the  controller for the observer rather than the control  plant itself.
Inspired by these ideas, the proposed approach can also be extended straightforwardly   to stabilize other distributed parameter systems with unknown control coefficient such as the wave equation and the beam equation.

%\noindent {\bf Hongyinping Feng }  received the B.Sc. degree, the
%M.Sc. degree,   and the Ph.D. degree, in mathematics  in 2003, 2006,
%2013, respectively, all from Shanxi University, P.R. China. He was a
%visiting scholar  at University of the Witwatersrand, South Africa
%from 2013 to 2014. He is currently a lecturer in the School of
%Mathematics Science, Shanxi University, P.R. China. His research
%interests focus on distributed parameter systems control.
%\noindent {\bf Xiao-Hui Wu} received the B.Sc. degree in
%mathematics from Shanxi University, Taiyuan,
%Shanxi, P.R.China in 2014, and is currently pursuing the Ph.D degree at Shanxi University,Taiyuan,
%Shanxi, P.R.China. His current research interests focus on distributed
%parameter systems control.

\section{Appendix}

\begin{lemma}\label{Lm20231004}
 Let $\omega>0$.   Suppose that the  functions $\gamma_1,\gamma_2\in L^2(0, \infty)$ and  $\gamma_3\in L^{\infty}(0,\infty)$. If the nonnegative
 %, absolutely continuous
 function $\eta$ satisfies
the differential inequality
%   Then, for any initial data $\eta_0\in\R$, the
%      solution  of   ODE     system
    \begin{equation}\label{202310041509}
\dot{\eta}(t)\leq -\omega\eta(t)+[\gamma_1(t)+\gamma_3(t)]\gamma_2(t) ,
\end{equation}
   then    $\eta(t)\to 0$ as $t\to+\infty$.
  \end{lemma}
\begin{proof}
By the Cauchy's inequality,
  \begin{equation*}\label{202310041359}
  |[\gamma_1(t)+\gamma_3(t)]\gamma_2(t)|\leq \frac{1}{2}\gamma_1^2(t)+\frac{1}{2}\gamma_2^2(t)+ |\gamma_2 (t)\gamma_3 (t)|.
  \end{equation*}
 If we let  $   g_1(t)=  |\gamma_2 (t)\gamma_3 (t)|$,
 $   g_2(t)=\frac{1}{2}\gamma_1^2(t)+\frac{1}{2}\gamma_2^2(t)$,
 then $g_1\in L^2(0,\infty)$ and $g_2\in L^1(0,\infty)$. Furthermore,
   \begin{equation}\label{202310011537}
\disp \dot{\eta} (t)\leq  -\omega\eta (t)+g_1(t)+g_2(t).
\end{equation}
% Letting   with $\eta(t)=\eta_1(t)+\eta_2(t)$, we divide  \dref{202310011537}  into two   systems:
%    \begin{equation}\label{202381919A}
%    \left.\begin{array}{l}
%\disp \dot{\eta}_j(t)\leq -\omega\eta_j(t)+g_j(t),
%%\disp \dot{\eta}_2(t)= -\omega\eta_2(t)+g_2(t).
%\end{array}\right.\ \ j=1,2.
%\end{equation}
  % Note that   \dref{20238231205} and \dref{2023941442f}, we have $f\in L^p(0,\infty)$, $p=1$ or $2$.
  Since $g_1\in L^2(0,\infty)$,
for any $\sigma >0$,   there exists a  $t_1>0$ such that $\|g_1\|_{L^2(t_1,\infty )}<\sigma$.
%\begin{equation}\label{2016112491020}
%\left.\begin{array}{l}
%\disp  \|g_1\|_{L^2(t_1,\infty )}<\sigma.
%\end{array}\right.
%\end{equation}
 It follows from the H\"{o}lder's inequality that
 \begin{equation}\label{20239202222}
 \int_{t_1}^te^{-\omega(t-s)} g_1 (s)ds  \leq \left( \int_{t_1}^te^{-2\omega(t-s)}ds\right)^{1/2}
\left( \int_{t_1}^tg_1^2  (s)ds \right)^{1/2}\leq \frac{\|g_1\|_{L^{2}(t_1,\infty  )}}{\sqrt{2\omega}}
 <  \frac{\sigma}{\sqrt{2\omega}}.
\end{equation}
   Since $g_2\in L^1(0,\infty)$,
for any $\sigma >0$,   there exists a  $t_2>0$ such that $\|g_2\|_{L^1(t_2,\infty)}<\sigma$,
%\begin{equation}\label{2016112491020}
%\left.\begin{array}{l}
%\disp  \|g_2\|_{L^1(t_2,\infty)}<\sigma,
%\end{array}\right.
%\end{equation}
which means that
 \begin{equation}\label{APPP6}
 \int_{t_2}^te^{-\omega(t-s)} g_2 (s)ds  \leq  \int_{t_2}^t
g_2  (s)ds
 \leq  \|g_2\|_{L^{1}(t_2,\infty  )}<\sigma.
\end{equation}
Choosing $t_0=\max\{t_1,t_2\}$ and taking \dref{20239202222} and \dref{APPP6} into account,  we solve   \dref{202310011537} to get
  \begin{eqnarray}\label{APPP4}
  \begin{array}{ll}
  \eta ( t)
&\disp \leq  e^{-\omega  t}  \eta (0)
+e^{-\omega(t-t_0)}\int_0^{t_0}e^{-\omega(t_0-s)} [ g_1  (s)+g_2(s)]ds
 +  \int_{t_0}^te^{-\omega(t-s)}  [ g_1  (s)+g_2(s)]ds\crr
 &\disp \leq  e^{-\omega  t}  \eta (0)
+e^{-\omega(t-t_0)}\int_0^{t_0}e^{-\omega(t_0-s)} [ g_1  (s)+g_2(s)]ds
 +   \left(\frac{1}{\sqrt{2\omega}}+1\right)\sigma.
 \end{array}
\end{eqnarray}
 Passing to the limit as $t \to+\infty $ in \dref{APPP4}, we obtain
 \begin{eqnarray}\label{APPP5B}
 \lim_{t\to+\infty} \eta (t)\leq  \left(\frac{1}{\sqrt{2\omega}}+1\right)\sigma,
\end{eqnarray}
 which leads to  $\eta(t)\to 0$ as $t\to+\infty$   due to   the arbitrariness of $\sigma$.
 \end{proof}

\end{document}